\documentclass[12pt,reqno]{amsart}

\usepackage{color}
\usepackage{amsmath, amsthm, amssymb}
\usepackage{amsfonts}
\usepackage[ansinew]{inputenc}
\usepackage[dvips]{epsfig}
\usepackage{graphicx}
\usepackage[english]{babel}
\usepackage{hyperref}

 \usepackage{mathrsfs}

\newcommand{\scrP }{\mathscr{P}}
\newcommand{\scrS }{\mathscr{S}}

\usepackage{cite}
\usepackage{graphicx}
\usepackage{amscd}
\usepackage{color}
\usepackage{bm}
\usepackage{enumerate}

\usepackage{verbatim}
\usepackage{hyperref}
\usepackage{amstext}
\usepackage{latexsym}
\theoremstyle{plain}
\newtheorem{theorem}{Theorem}[section]

\newtheorem{proposition}[theorem]{Proposition}
\newtheorem{lemma}[theorem]{Lemma}
\newtheorem{remark}[theorem]{Remark}

\numberwithin{theorem}{section}
\numberwithin{equation}{section}

\newcommand{\average}{{\mathchoice {\kern1ex\vcenter{\hrule height.4pt
width 6pt depth0pt} \kern-9.7pt} {\kern1ex\vcenter{\hrule
height.4pt width 4.3pt depth0pt} \kern-7pt} {} {} }}

\def\R{\mathbb{R}}

\newcommand{\va }{\varphi }

\renewcommand{\a }{\alpha }
\renewcommand{\b }{\beta }
\renewcommand{\d}{\delta }

\newcommand{\e }{\varepsilon }

\newcommand{\G }{\Gamma}
\renewcommand{\l }{\lambda }
\renewcommand{\L }{\Lambda }

\newcommand{\vp }{\varphi }

\newcommand{\s }{\sigma }

\renewcommand{\t }{\tau }
\newcommand{\z }{\zeta}
\renewcommand{\th }{\theta }

\newcommand{\ov}{\overline}

\newcommand{\be}{\begin{equation}}
\newcommand{\ee}{\end{equation}}

\newcommand{\de}{\partial}

\newcommand{\ti}{\widetilde}

\newcommand{\ra}{{\rangle}}
\newcommand{\la}{{\langle}}

\newcommand{\calO }{\mathcal{O}}

\newcommand{\N}{\mathbb{N}}

\newcommand{\cD}{{\mathcal D}}

\newcommand{\cF}{{\mathcal F}}

\newcommand{\cH}{{\mathcal H}}

\newcommand{\cK}{{\mathcal K}}
\newcommand{\cL}{{\mathcal L}}
\newcommand{\cM}{{\mathcal M}}
\newcommand{\cN}{{\mathcal N}}
\newcommand{\cO}{{\mathcal O}}

\newcommand{\eps}{\varepsilon}

\renewcommand{\phi}{\varphi}

\renewcommand{\epsilon}{\varepsilon}

\begin{document}

\title[Delaunay hypersurfaces with constant nonlocal mean curvature]
{Delaunay hypersurfaces with constant nonlocal mean curvature }

\author{Xavier Cabr\'e}
\address{X. Cabr\'e\textsuperscript{1,2}:
\newline
\textsuperscript{1} Universitat Polit\`ecnica de Catalunya, Departament
de Matem\`{a}tiques, Diagonal 647, 08028 Barcelona, Spain
\newline
\textsuperscript{2} ICREA, Pg. Lluis Companys 23, 08010 Barcelona, Spain}
\email{xavier.cabre@upc.edu}

\author[Mouhamed M. Fall]
{Mouhamed Moustapha Fall}
\address{M. M. Fall: African Institute for Mathematical Sciences of Senegal, 
KM 2, Route de Joal, B.P. 14 18. Mbour, S\'en\'egal}
\email{mouhamed.m.fall@aims-senegal.org}

\author{Tobias Weth}
\address{T. Weth: Goethe-Universit\"{a}t Frankfurt, Institut f\"{u}r Mathematik.
Robert-Mayer-Str. 10 D-60054 Frankfurt, Germany}
\email{weth@math.uni-frankfurt.de}

\thanks{The first author is supported by MINECO grant MTM2014-52402-C3-1-P. He
is member of the Barcelona Graduate School of Mathematics and of the Catalan research group 2014 SGR 1083.
 The second author's work is supported by the Alexander 
von Humboldt foundation. Part of this work was done while he was visiting
the University of Frankfurt, during July and August, 2015. The third author is supported by DAAD (Germany) within the program 57060778.}


\begin{abstract}
We study hypersurfaces of $\mathbb{R}^N$ with constant nonlocal (or fractional) mean curvature.
This is the equation associated to critical points of the fractional perimeter functional under a volume constraint. 
We establish the existence of a smooth branch of periodic cylinders in $\mathbb{R}^N$, $N\geq 2$, 
all of them with the same constant nonlocal mean curvature, and bifurcating from 
a straight cylinder. These are Delaunay type cylinders in the nonlocal setting. The proof uses 
the Crandall-Rabinowitz theorem applied to a quasilinear type fractional elliptic equation.\\

\bigskip
\noindent

R\'{E}SUM\'{E}. Nous \'{e}tudions des hypersurfaces dans $\R^N$, $N \geq  2$,  \`{a}  courbure moyenne non-locale (ou fractionaire) constante. Cel\`{a} revient  \`{a} \'{e}tudier une \'{e}quation associ\'{e}e aux points critiques du p\'{e}rim\`{e}tre fractionnaire sous une contrainte de volume. Nous \'{e}tablissons l'existence d'une branche lisse d'hypersurfaces p\'{e}riodiques de type Delauney qui ont toutes la m\^{e}me courbure moyenne non-locale que celle d'un cylindre droit.  La preuve utilise le th\'{e}or\`{e}me de bifurcation de Crandall-Rabinowitz appliqu\'{e}  \`{a} une \'{e}quation elliptique fractionnaire de type quasilin\'{e}aire.

\end{abstract}

\maketitle

\section{Introduction and main results}
\label{sec:main-result}
Let $\alpha \in (0,1)$, $N\geq 2$,  and let $E$ be an open set  in $\R^N$ with $C^2$-boundary. 
For every $x \in \partial E$, the nonlocal or fractional mean curvature
of $\partial E$ at $x$ (that we call NMC for short) is given by
\begin{equation}
  \label{eq:def-frac-curvature}
H_E(x)= PV \int_{\R^N} \frac{1_{E^c}(y) -1_{E}(y)}{|x-y|^{N+\alpha}}\,dy :=
\lim_{\eps \to 0} \int_{|y-x| \ge \eps} \frac{1_{E^c}(y) -1_{E}(y)}{|x-y|^{N+\alpha}}\,dy  
\end{equation}
and is well defined. Here and in the following, $E^c$ denotes the complement of $E$ in  $\R^N$ and $1_A$ denotes the characteristic function of a set $A \subset \R^N$. 
In the first integral PV denotes the principal value sense.
For the asymptotics $\alpha$ tending to 0 or 1, $H_E$ should be
renormalized with a positive constant factor $C_{N,\alpha}$.
Since constant factors are not relevant for the results of this paper, we use the simpler expression in 
(\ref{eq:def-frac-curvature}) without the constant~$C_{N,\alpha}$.

An alternative expression for the NMC is given by   
\begin{equation}\label{geometric-0}
H_E(x)= -\frac{2}{\alpha} \int_{\partial E} |x-y|^{-(N+\alpha)} (x-y)\cdot\nu_E(y)\,dy, 
\end{equation}
where $\nu_E(y)$ denotes the outer unit normal to $\partial E$ at $y$. If $\partial E$ is of class 
$C^{1,\beta}$ for some $\beta>\alpha$ and $\int_{\partial E}(1+ |y|)^{1-N-\alpha}\,dy< \infty$, then the integral in 
\eqref{geometric-0} is absolutely convergent in the Lebesgue sense and the expression follows from 
\eqref{eq:def-frac-curvature} via the divergence theorem.

The notion of nonlocal mean curvature was introduced around 2008 by Caffarelli and Souganidis in \cite{Caff-Soug2010} and by
Caffarelli, Roquejoffre, and Savin in~\cite{Caffarelli2010}. As first discovered in \cite{Caffarelli2010}, the nonlocal mean curvature arises as the first variation of the fractional perimeter. 
For the notion of fractional perimeter and its convergence to classical perimeter as $\alpha \to 1$, see the papers \cite{Davila02,Ponce04,Ambrosio}. The seminal paper \cite{Caffarelli2010} established the first
existence and regularity theorems on nonlocal minimal surfaces, that is, (minimizing) hypersurfaces with zero NMC.  
Within these years, there have been important efforts and results concerning
nonlocal minimal surfaces but still, apart from dimension $N=2$, there is a lot to be understood,  
mainly for the classification of stable nonlocal minimal cones. See \cite[Chapter 6]{Bucur} for a recent survey of known results.

The purpose of this paper is to establish a nonlocal analogue of the classical result of
Delaunay \cite{Delaunay} on periodic cylinders with constant mean curvature, the so called onduloids.
In \cite{Cabre2015A}, a paper by the present authors and Sol\`a-Morales, we accomplished this in the plane $\R^2$;
that is, we proved the existence of a continuous branch of periodic bands, starting from a straight band, all of them
with the same constant NMC. Here we establish the analogue result but in $\R^N$ for $N\geq 3$. 
In addition, we show that the branch is not only continuous but smooth ---and we prove this also in $\R^2$.

More precisely,
we consider sets $E\subset\R^N$, $N\geq 2$, with constant nonlocal mean curvature
which have the form 
\begin{equation}
  \label{eq:cylindrical-graph}
E_u=\{(s,\z)\in \R\times \R^{N-1}\, : \, |\z|<u (s)\},
\end{equation}
where $u: \R \to (0,\infty)$ is a positive function. 
We establish the existence of a smooth branch of sets as above (that we call cylinders) 
which are periodic in the variable~$s$ and have
all the same constant nonlocal mean curvature; they bifurcate from a straight cylinder $\{|\z|<R\}$. 
The radius $R$ of the straight cylinder is chosen so that the periods of the new cylinders converge to $2\pi$ as they
approach the straight cylinder. Our result is of perturbative nature and thus 
we find periodic cylinders which are all close to the straight one.  

The following is the precise statement of our result. Throughout the paper, $C^{k,\gamma}(\R)$ denotes the space of
$C^k(\R)$ bounded functions $u$, with bounded derivatives up to order $k$ and with $u^{(k)}$ having finite H\"older 
seminorm of order $\gamma\in (0,1)$. The space is equipped with the standard norm \eqref{eq:def-hoelder-norm}.

\begin{theorem}
\label{res:cyl1}
Let $N \ge 2$. For every $\alpha\in(0,1)$ there exist $R>0$, $a_0>0$, $\beta \in (\alpha,1)$ and $C^\infty$-maps 
\begin{align*}
&(-a_0,a_0) \to C^{1,\beta}(\R), \qquad a \mapsto u_a\\
&(-a_0,a_0) \to (0,\infty), \qquad a \mapsto \lambda(a)  
\end{align*}
with the following properties:
\begin{itemize}
\item[(i)] $\lambda(0)=1$ and $u_0 \equiv R$. 
\item[(ii)] For every $a \in (-a_0,a_0) \setminus \{0\}$, the function $u_a: \R \to \R$ is even and periodic with minimal period ${2\pi}/{\lambda(a)}$. 
\item[(iii)] For every $a \in (-a_0,a_0)$, the set 
$$
 E_{u_a}=\{(s,\z)\in \R\times \R^{N-1}\, : \, |\z|<u_a (s)\}
$$ 
has positive constant nonlocal mean curvature equal to the nonlocal mean curvature of the straight cylinder 
$$
E_R:= \{(s,\z)\in \R\times \R^{N-1}\, : \, |\z|<R\}.
$$
Moreover, $ E_{u_a}\not= E_{u_{a'}}$ for $a,a' \in (-a_0,a_0)$, $a\not= a'$.
\item[(iv)] For every $a \in (-a_0,a_0)$, we have 
\begin{equation}\label{form}
  u_{a}(s)=R+\frac{a}{\l(a)}\left\{\cos\left( \lambda(a) s\right)+ v_{a}(\lambda(a)s)\right\},
\end{equation}
where $v_{a}\to 0$ in $C^{1,\beta}(\R)$ as $a\to 0$ and $\int_{-\pi}^{\pi}v_{a}(t)\cos(t)\, dt=0$ for every $a \in (-a_0,a_0)$. Moreover, we have 
$$
\lambda(-a)= \lambda(a)\quad \text{and}\quad u_{-a}(s)= u_a\Bigl(s+ \frac{\pi}{\lambda(a)}\Bigr)
$$
for every $a \in (-a_0,a_0)$, $s \in \R$, $\lambda'(0)=0$ and $\partial_a u_a \big|_{a=0}  = \cos(\cdot)$.  
\end{itemize}
\end{theorem}

We prove that the branch is  $C^\infty$ in the parameter $a$, extending our previous work \cite{Cabre2015A} in $\R^2$ where we only proved continuous dependence. 

The smoothness (i.e., the $C^\infty$-character) of our $C^{1,\beta}$ hypersurfaces $\de E_{u_a}$, and in general of 
$C^{1,\beta}$ hypersurfaces in $\R^N$ with constant NMC which are, locally, Lipschitz graphs follows (since $\beta >\alpha$)
from the methods and results of Barrios, Figalli, and Valdinoci~\cite{Barrios} on nonlocal minimal graphs. 
This holds for all $N\geq 2$.
More generally, to deduce the $C^\infty$ regularity, \cite{Barrios} needs to assume that the hypersurface is 
$C^{1,\beta}$ for some $\beta>\alpha/2$  and that it has constant 
nonlocal mean curvature in the viscosity sense; this fact can be found in Section~3.3
of \cite{Barrios}. Here, the notion of viscosity solution is needed since the expression 
\eqref{eq:def-frac-curvature} for the NMC is only well defined
for $C^{1,\beta}$ sets when $\beta > \alpha$.

Regarding CNMC hypersurfaces, that is, hypersurfaces with constant nonlocal
mean curvature, there have been three articles before this one (apart from the papers on 
zero NMC, that is, nonlocal minimal surfaces).
In \cite{Cabre2015A}, besides finding the Delaunay bands in $\R^2$, the present authors and Sol\`a-Morales also
established the analogue of the Alexandrov rigidity theorem for bounded CNMC hypersurfaces in $\R^N$; these sets
must be balls. At the same time and independently, Ciraolo, Figalli, Maggi, and Novaga~\cite{Ciraolo2015} also proved
the Alexandrov rigidity theorem for CNMC hypersurfaces in $\R^N$ and, in addition, a strong quantitative version
of this rigidity theorem.

A third paper, \cite{Davila2015}, by D\'avila, del Pino, Dipierro, and Valdinoci, establishes variationally
the existence of periodic and cylindrically symmetric hypersurfaces in $\R^N$
which minimize a certain renormalized fractional perimeter under a volume constraint.
More precisely, \cite{Davila2015}~establishes the existence of a $1$-periodic minimizer for every given volume within the
slab $\{(s,\z)\in\R\times\R^{N-1}\, : \, -1/2<s<1/2\}$.
We have realized recently that these minimizers are in fact CNMC hypersurfaces in a certain weak sense. 
They would be CNMC hypersurfaces in the classical sense defined above if one could prove
that they are of class $C^{1,\beta}$ for some $\beta>\alpha/2$.
However, \cite{Davila2015} does not prove any regularity
for the minimizers. The article also proves that for small volume constraints, the minimizers
tend in measure (more precisely, in the so called Fraenkel asymmetry) to a
periodic array of balls.

It is an open problem to establish the existence of global continuous 
branches of nonlocal Delaunay hypersurfaces  as in Theorem~\ref{res:cyl1}
and to study their limiting configuration. In the case of classical mean curvature, embedded Delaunay hypersurfaces 
vary from a cylinder to an infinite compound of tangent spheres. However, it is easy to see that 
an infinite compound of aligned round spheres, tangent or disconnected, does not have constant NMC. 
In a forthcoming paper~\cite{Cabre2015B}, we study nonlocal analogues of this periodic and disconnected CMC set.

Also related to our work, the papers \cite{Sicbaldi2010} and \cite{Sicbaldi2012}
established the existence of periodic and cylindrical symmetric domains in $\R^N$ whose first Dirichlet
eigenfunction has constant Neumann data on the boundary. This is therefore a
nonlinear and nonlocal elliptic operator of order 1 based on a certain Dirichlet to Neumann map.

The nonlocal mean curvature flow for the notion of 
NMC considered in this paper has been studied in strong sense in \cite{Saez2015} and in viscosity sense
in \cite{Imbert2009,Chambolle.Morini.Ponsiglione2015,Cinti.Sinestrari.Valdinoci-ArXiv}. 

Let us describe the proof of Theorem~\ref{res:cyl1} and its main difficulties.
The first step is to write the NMC operator acting on graphs 
of functions ---the functions $u_a$ above. This leads to an integral operator of quasilinear type 
acting on functions $u=u(s)$ and involving a double integral with respect to $d\sigma ds$, where $\sigma\in S^{N-2}$ 
takes into account the symmetry of revolution in the variable $\z\in \R^{N-1}$. The presence of this new integral in $d\sigma$
is the main difference and difficulty with respect to our previous paper where $N=2$. In fact, changing the order of integration in
$d\sigma ds$, or making different changes of variables to simplify the integrands, will lead to quite different expressions
for the nonlocal mean curvature of the set $E_a$. We will present three of such expressions, namely
\eqref{new-rep-H}, \eqref{eq:good-expression-NMC}, and \eqref{eq:NMC-varphi} below. Finding the {\em second} of these
expressions was crucial to be able to prove the smoothness of the nonlocal mean curvature operator, which is stated in  
Proposition~\ref{lem:diff-H} and established in Proposition~\ref{prop:smooth-ovH} below.

Another essential point in the proof is to have a simple expression for the linearized operator at the straight
cylinder. This is given in Proposition~\ref{lem:diff-H} and we found it using our {\em third} expression \eqref{eq:NMC-varphi}
for $H$. Even though we prove the formula for the
linearized operator by using our second expression for $H$ as
given in \eqref{eq:good-expression-NMC}, it would have been very difficult to guess it from this second expression.

The linearization gives rise to an integro-differential operator with a singular kernel close to (but different than) that
of the fractional Laplacian. This is another difference with the previous 2D case. We use regularity theory
both in Sobolev and H\"older spaces to analyze the linear operator and, thus, to be able to apply
the Crandall-Rabinowitz theorem in~\cite{Crandall1971}, which will lead to our result.

The paper is organized as follows. Section~2 sets up the nonlinear nonlocal operator to be studied and states a simple
expression for the linearized 
operator at a straight cylinder. It contains also some preliminary estimates concerning the linearized problem. 
These estimates 
are used in Section~3 to solve our nonlinear problem using the Crandall-Rabinowitz theorem. 
In Section~4 we establish the $C^\infty$ character of our nonlocal mean curvature operator
and we prove the formula for the linearized operator at a straight cylinder. Since some expressions and
estimates in the previous sections require $N\geq 3$, in Section~5 we treat the case $N=2$.

\section{The NMC operator acting on cylindrical graphs of $\mathbb{R}^N$}
Let $\alpha \in (0,1)$ and $\beta \in (\alpha,1)$. In Section~\ref{sec:nonl-probl-solve} 
we will also assume that $\beta < 2\alpha + 1/2$; see (\ref{defbeta}) below. 
This extra assumption will only be used at the end of the proof of Proposition~\ref{propPhi}.   
For a positive function $u \in C^{1,\beta}(\R)$, we consider the set $E_u$ as defined in (\ref{eq:cylindrical-graph}).  
We first recall the following expression for the NMC of $E_u$:
\begin{equation}\label{geometric}
H_{E_{u}}(x)= -\frac{2}{\alpha} \int_{\partial E_{u}} |x-y|^{-(N+\alpha)} (x-y)\cdot\nu_{E_{u}}(y)\,dy; 
\end{equation}
see e.g. \cite[Eqn. (1.2)]{Cabre2015A}. Here $\nu_{E_{u}}(y)$ denotes the unit outer normal of $ E_{{u}}$ and $dy$ is the volume element of $\de E_{{u}}$. Next, we consider the open set
\begin{equation}
  \label{eq:def-cO}
\cO:= \{{u}  \in C^{1,\beta}(\R) \::\: \inf_\R {u}  >0\}.
\end{equation}
For ${u} \in \cO$, we consider the map $F_{u}:  \R\times \R^{N-1} \to \R^N$ given by 
$$
F_{u}(s,\z)= (s,{u}(s) \z).
$$
We have that the boundary of $E_u$, 
$$
\de E_{u}=\left\{(s,{u}(s)\s) \in \R\times \R^{N-1}\, : \, \s \in S^{N-2}\right\},
$$
is parameterized  by the restriction of  $F_u $ to  $\R\times S^{N-2}$.

\subsection{Two fundamental expressions for the NMC operator}
The following results provide two expressions for the NMC of $E_{u}$ in terms of the above parametrization and the 
function $u$. Here, when $N=2$, we have $S^{N-2}=S^0=\{-1,1\}\subset\R$.
\begin{lemma}\label{lem:geomNMC}  
Let ${u}\in \cO$. Then the nonlocal mean curvature $H_{E_{u}}$  ---that we will denote by $H({u})(s)$---
at a point $(s,{u}(s)\th )$, with $\th\in S^{N-2}$, does not depend on $\th$ and is given by 
\begin{align}
  -\frac{\alpha}{2}  H({u})(s) =& \int_{S^{N-2}}  \int_{\R}
  \frac{
\bigl\{ u(s)-u(s-\tau)-\tau u'(s-\tau) \bigr\}u^{N-2} (s-\t)
}{
\{\t^2+({u}(s)-{u}(s-\t))^2+{u}(s){u}(s-\t)|\sigma-e_1|^2 \}^{(N+\a)/2}
}
d\tau d\s   \nonumber \\
& \hspace{-14mm} -
 \frac{{u}(s)}{2} \int_{S^{N-2}}   \int_{\R}
 \frac{
|\sigma-e_1|^2 u^{N-2} (s-\t)
}{
\{\t^2+({u}(s)-{u}(s-\t))^2+{u}(s){u}(s-\t)|\sigma-e_1|^2 \}^{(N+\a)/2}
}
d\t d\s  , \label{new-rep-H}
\end{align} 
where $e_1=(1,0,\dots,0) \in \R^{N-1}$.
Moreover, when $N\geq 3$ the two integrals above converge absolutely in the Lebesgue sense.
\end{lemma}
\begin{proof}
Starting from the absolutely converging integral representation \eqref{geometric}, we deduce that 
\begin{align}
  -\frac{\alpha}{2}  H({u})(s)& = -\frac{\alpha}{2} H_{E_{u}}( F_{u}(s,\th ))\nonumber\\
&=  \int_{S^{N-2}} \int_{\R} \frac{\left\{  F_{u}(s,\th)-F_{u}(\bar{s},\s)    \right\} \cdot\nu_{E_{u}}(F_{u}(\bar{s},\s))  }{ |F_{u}(s,\th)-F_{u}(\bar{s},\s)|^{N+\alpha}}  J_{F_{u}}(\bar{s},\s)\,d \bar s d \sigma, \label{eq-1-H(u)}
 \end{align}
where the unit outer normal of   $\de E_{u}$ at the point ${F_{u}}(\bar{s},\s)$ is given by
 $$
\nu_{E_{u}}(F_{u}(\bar{s},\s))=\frac{1}{\sqrt{1+({u}')^2(\bar{s})}}(-{u}'(\bar{s}),\s)\qquad \textrm{ for } \bar{s}\in\R,\, \s\in S^{N-2},
 $$
and
 $$
 J_{F_{u}}(\bar{s},\s)=\sqrt{{u}^{2(N-2)}(\bar{s})(1+({u}')^2(\bar{s}))}={u}^{N-2}(\bar{s}) \sqrt{1+({u}')^2(\bar{s})}  \quad \textrm{ for } \bar{s}\in\R,\, \s\in S^{N-2}.
 $$
We also note that for $s, \bar s \in \R$ and $\th, \s \in S^{N-2}$ we have  
\begin{align*}
|F_{u}(s,\th)-F_{u}(\bar{s},\s)|^2 &= (s-\bar s)^2 + | u(s)\th - u(\bar s) \s|^2\\
& = (s-\bar{s})^2+ (u(s)-u(\bar s))^2 +2 u(s)u(\bar s)(1-   \theta \cdot \sigma ) 
\end{align*}
and 
\begin{align*}
\left\{  F_{u}(s,\th)-F_{u}(\bar{s},\s)    \right\} \cdot\nu_{E_{u}}(F_{u}(\bar{s},\s)) &= \frac{-(s-\bar s)u'(\bar s)+(u(s)\th - u(\bar s) \s) \cdot \s }{\sqrt{1+({u}')^2(\bar{s})}}\\
&=  \frac{u(s)-{u}(\bar{s})-(s-\bar s) u'(\bar s)- u(s)(1-\th \cdot\sigma)}{\sqrt{1+({u}')^2(\bar{s})}}.
\end{align*}
Inserting these identities in \eqref{eq-1-H(u)}, we obtain 
 \begin{align*}
-\frac{\alpha}{2}&  H({u})(s)=\int_{S^{N-2}} \int_{\R}\frac{\left\{u(s)-{u}(\bar{s})-(s-\bar s) u'(\bar s)- u(s)(1-\th \cdot\sigma)\right\} u^{N-2}(\bar s)}{\{(s-\bar{s})^2+ (u(s)-u(\bar s))^2 +2 u(s)u(\bar s)(1- \theta \cdot \sigma  )\}^{\frac{N+\alpha}{2}}}d\bar{s} d\s\\
&=\int_{S^{N-2}} \int_{\R}\frac{\left\{u(s)-{u}(\bar{s})-(s-\bar s) u'(\bar s)- u(s)(1-\sigma_1)\right\} u^{N-2}(\bar s)}{\{(s-\bar{s})^2+ (u(s)-u(\bar s))^2 +2 u(s)u(\bar s)(1- \sigma _1)\}^{\frac{N+\alpha}{2}}}d\bar{s} d\s\\
&=\int_{S^{N-2}} \int_{\R}\frac{\left\{u(s)-{u}(s-\tau)-\tau u'(s-\tau)- u(s)(1-\sigma_1)\right\} u^{N-2}(s-\tau)}{\{\tau^2+ (u(s)-u(s-\tau))^2 +2 u(s)u(s-\tau)(1- \sigma _1)\}^{\frac{N+\alpha}{2}}}d\tau d\s.
 \end{align*}
Here, for the second equality, we note that the rotation invariance of the spherical integral allows to choose $\th = e_1 \in S^{N-2}$, whereas the third equality follows from the change of variable $\tau= s-\bar s$. 

Since $1-\sigma_1 = \frac{|\sigma-e_1|^2}{2}$ for $\sigma \in S^{N-2}$, the assertion of the lemma now follows once we have shown that both integrals in \eqref{new-rep-H} converge absolutely in the Lebesgue sense. To prove this, we first note that
\begin{align*}
| u(s)-{u}(s-\tau)-\tau u'(s-\tau) | &\leq |\tau| \int_0^1 |u'(s-\rho \tau)-u'(s-\tau)|\,d\rho \\
&\le  2 \|u\|_{C^{1,\b}(\R)}\,\min( |\tau|^{1+\b}, |\tau|) 
 \end{align*}
for $s,\tau \in \R$. Using this, we get   
\begin{align*}
&\int_{S^{N-2}} \int_{\R}\frac{|u(s)-{u}(s-\tau)-\tau u'(s-\tau)| u^{N-2}(s-\tau)}{\{\tau^2+ (u(s)-u(s-\tau))^2 + u(s)u(s-\tau)|\s-e_1 |^2\}^{\frac{N+\alpha}{2}}}d\tau d\s\\
&\le 2 \|u \|_{C^{1,\b}(\R)}^{N-1} \int_{S^{N-2}} \int_{\R}\frac{\min( |\tau|^{1+\b}, |\tau|)}{\{\tau^2+ (u(s)-u(s-\tau))^2 + u(s)u(s-\tau)|\s-e_1 |^2\}^{\frac{N+\alpha}{2}}}d\tau d\s\\
&\le 2 \|u \|_{C^{1,\b}(\R)}^{N-1} \int_{S^{N-2}} \int_{\R}\frac{\min( |\tau|^{1+\b}, |\tau|)}{(\tau^2+ \delta^2 |\s-e_1 |^2)^{\frac{N+\alpha}{2}}}d\tau d\s 
\end{align*}
with $\delta := \inf_{\R} u>0$. Since $N \ge 3$, the change of variable $\tau =|\s-e_1 | t$ now leads to 
\begin{align*}
\int_{S^{N-2}} \int_{\R}\frac{|\tau|^{1+\b}}{(\tau^2+ \delta^2 |\s-e_1 |^2)^{\frac{N+\alpha}{2}}}d\tau d\s
&= \int_{S^{N-2}} \frac{d\s}{|\s-e_1 |^{N+\a-2-\b}} \int_{\R} 
\frac{|t|^{1+\b}}{(t^2+\delta^2)^{\frac{N+\alpha}{2}}}dt < \infty.  
\end{align*}
Hence the first integral in \eqref{new-rep-H} converges absolutely. 

To see the absolute convergence of the second integral in \eqref{new-rep-H}, we again use the change of variable $\tau =|\s-e_1 | t$ to obtain the estimate
\begin{align*}
\int_{S^{N-2}} &\int_{\R}\frac{|\s-e_1|^2 u^{N-2}(s-\tau)}{\{\tau^2+ (u(s)-u(s-\tau))^2 + u(s)u(s-\tau)|\s-e_1|^2\}^{\frac{N+\alpha}{2}}}d\tau d\s\\
&\le \|u \|_{C^{1,\b}(\R)}^{N-2} \int_{S^{N-2}} \int_{\R}\frac{|\s-e_1|^2}{(\tau^2+ \delta^2 |\s-e_1|^2)^{\frac{N+\alpha}{2}}}d\tau d\s\\
&= \|u \|_{C^{1,\b}(\R)}^{N-2}   \int_{S^{N-2}}\frac{d\s}{|\s-e_1 |^{N+\a-3}}\, \int_{\R} \frac{dt }{    (t^2+ \d^2)^{\frac{N+\a}{2}} } <\infty.
\end{align*}
The proof is finished.
\end{proof}

To prove the smoothness of the nonlocal mean curvature operator between appropriate H\"older spaces, 
it will be crucial to make a further transformation in the expression of  
$H$ found in the previous Lemma~\ref{lem:geomNMC}. To describe this, we first introduce some notation. We denote
$$ 
p_\s:=|\s-e_1|,
$$
and  for $r\in\R$, we define 
$$
\mu_r(\s)= \frac{1}{|\s-e_1|^{N+r}} = p_\s^{-N-r}.
$$
We define the  maps $\L_0, \L: C^{1,\b}(\R)\times\R\times\R\times \R\to \R $ by   
$$
\L_0(\phi,s,t,p) = \frac{\phi(s)-\phi(s-pt)}{pt} = \int_0^1 \phi'(s-\rho p t) d\rho 
$$
and 
$$
\L(\phi, s,t, p)= \L_0(\phi,s,t,p)-  \phi'(s-pt) =  \int_0^1(\phi'(s-\rho p t)-\phi'(s-p t) ) d\rho.
$$

Recalling that  $ 2(1-\s\cdot  e_1 )=|\s- e_1|^2=p_\s^2$, we make the change of variables 
$$
 t=\frac{ \t}{   |\s-e_1|}= \frac{ \t}{   p_\s}
$$
in the expression  for $H({u})$ in Lemma \ref{lem:geomNMC}. We immediately obtain

\begin{lemma}\label{Hreg}
With the notation above, for $u\in\mathcal{O}$, we have
\begin{align}\label{eq:good-expression-NMC}
-\frac{\a}{2} H({u})(s)
=&\int_{S^{N-2}}\mu_{\a-2}(\s) \int_{\R}  t\L ({u}, s,t, p_\s)\cK_\a({u}, s,t,p_\s){u}^{N-2}(s-p_\s t) dt  d\s  \\
&\hspace{5mm}- \frac{{u} (s)}{2}    \int_{S^{N-2}}\mu_{\a-3}(\s) \int_{\R}  \cK_\a({u},s,t,p_\s){u}^{N-2}(s-p_\s t) dt  d\s \nonumber, 
\end{align}
where the function  $\cK_\alpha: C^{1,\b}(\R)\times\R\times\R\times\R\to \R$ is defined by 
\begin{equation*}
\cK_\alpha({u},s,t,p)=\frac{1}{\left( \displaystyle t^2+ t^2 \L_0(u,s,t,p)^2+{u}(s){u}(  s-p t)  \right)^{(N+\alpha)/2}}.
\end{equation*}
Moreover, when $N\geq 3$ the two integrals above converge absolutely in the Lebesgue sense.
\end{lemma}

%

We point out a very important and useful difference between this last expression for $H$ and that of Lemma~\ref{lem:geomNMC}.
In \eqref{eq:good-expression-NMC} the dependence on the variable $|\sigma-e_1|=p_\sigma=p$ appears ``inside the 
known variables for $u$'', that is, through $u(s-pt)$. This will allow us to establish in the 
following Proposition~\ref{lem:diff-H} a fundamental result on the smoothness of the nonlocal mean curvature map $H$. The result also states an
expression for the differential of $H$ at a constant function 
(recall that a constant function corresponds to a straight cylinder in $\R^N$). 
The result will be proved further on in Section \ref{sec:preparations-1} (see Propositions~\ref{prop:smooth-ovH} 
and~\ref{cor-1st-der}) using the previous expression \eqref{eq:good-expression-NMC}. 
Recall the definition of the set $\calO \subset C^{1,\beta}(\R)$ defined in (\ref{eq:def-cO}).

\begin{proposition}\label{lem:diff-H} 
For  $N\geq 3$, the map $H: \calO \subset C^{1,\beta}(\R) \to C^{0, \b-\a}(\R)$ is of class $C^{\infty}$. In addition, if $\kappa \in \calO$ is a constant function, then we have 
 \be\label{eq:dHuv}
 DH(\kappa)v(s)= \kappa^{-1-\alpha}\Bigl(  PV \int_{\R} (v(s)-v(s-\kappa \tau)) G_\alpha(\tau)\,d\tau - b_\alpha v(s)\Bigr)  
  \ee
for $v \in C^{1,\beta}(\R)$, where 
$$
G_\alpha: \R \setminus \{0\} \to \R, \qquad G_\alpha(\tau)= 2\int_{S^{N-2}}\frac{1}{(\tau^2+2(1-\s_1))^{\frac{N+\a}{2}}}\,d\s
$$
and 
$$
b_\alpha = 2 \int_{\R} \int_{S^{N-2}} \frac{1-\sigma_1}{(\tau^2  + 2(1-\sigma_1))^{\frac{N+\alpha}{2}}}\,d\s d\tau.
$$
\end{proposition}

Expression \eqref{eq:dHuv} is rather simple and this will be crucial in order to analyze 
the associated linearized operator. We prove \eqref{eq:dHuv} in all detail in Proposition~\ref{cor-1st-der}
using expression \eqref{eq:good-expression-NMC}. The proof, however, involves collecting several groups
of terms that could not have been guessed without knowing apriori the expression \eqref{eq:dHuv} that we want to establish.
Indeed, we deduce the expression \eqref{eq:dHuv} for the linearized operator at a constant function from  another, very different, expression for the NMC operator $H$ that we describe next.

\subsection{A third expression for the NMC operator: finding the linearized operator}
The following is another formula for $H$. We present it here only to show how we found 
expression \eqref{eq:dHuv} for the linearized operator. It will not be used in any proof of the paper.
 
The  nonlocal mean curvature $H_{E_u}$
at a point $(s,u(s)\th )$, with $\th\in S^{N-2}$, does not depend on $\theta$ and is given by 
\be\label{eq:NMC-varphi}
H(u)(s)= u(s)^{-1-\alpha} \int_\R \left\{
2 I\left( \frac{u(\bar{s} )}{u(s)},\frac{s-\bar{s}}{u(s)}  \right)
-I\left( 0, \frac{s-\bar{s}}{u(s)}  \right)
\right\}  d\bar{s} ,
\ee 
where
$$
I(q,p):=\int_q^{+\infty} \int_{S^{N-2}}  \frac{\tau^{N-2}}{(p^2+1+\tau^2-2\sigma_1\tau )^{(N+\alpha)/2}}\,
d\sigma d\tau.
$$

Since we will not use \eqref{eq:NMC-varphi} in any proof of the paper, we merely sketch a proof of this formula 
without looking in detail at the convergence of integrals. 
We define $\mu(s,\z) =|(s,\z)|^{-N-\a}= (s^2 +|\z|^2)^{-\frac{N+\alpha}{2}}$ for $(s,\z) \in \R \times \R^{N-1}$. 
Let $s \in \R$ and $\th\in S^{N-2}$. At the point $(s, u(s)\theta) \in \de E_{u}$, we then have 
$$
H(u)(s)=\int_{\R}\int_{\R^{N-1}}\left\{1_{ E_{u}^c}(\bar{s},\z)-1_{ E_{u}}(\bar{s},\z) \right\}\mu\left( (s,\th u (s))-
(\bar{s},\z)\right) \, d\z
d\bar{s}.
$$
Using polar coordinates, we get  
$$
 H(u)(s)=\int_{\R}\int_{S^{N-2}} \left( \int_{u(\bar{s})}^{+\infty}-\int_0^{u(\bar{s})} \right)
 r^{N-2}\mu( (s,\th u (s))-(\bar{s},r\s))\, dr d\s d\bar{s}.
$$
By direct computation, we have
$$
\mu\left( (s,\th u (s))-(\bar{s},r\s)\right) =\left(  (s-\bar{s})^2+u^2(s)+r^2-2ru(s)\s\cdot\th \right)^{-\frac{N+\a}{2}}.
$$
Using this, we can see that the integral
$\int_{S^{N-2}} \mu ( (s,\th u (s))-(\bar{s},r\s)) \,d\s$
is independent of $\th$. Hence we may assume that $\th=e_1$. We have
\begin{align*}
H(u)(s)&=\int_{\R}\int_{S^{N-2}} \int_{u(\bar s)}^{+\infty}\frac{r^{N-2}}{\left(  (s-\bar{s})^2+u^2(s)+r^2-2ru(s)\s_1 \right)^{\frac{N+\a}{2}}} \, dr d\s d\bar{s}\\
& \hspace{10mm}-
\int_{\R}\int_{S^{N-2}}   \int_0^{u(\bar{s})}  \frac{r^{N-2}}{\left(  (s-\bar{s})^2+u^2(s)+r^2-2ru(s)\s_1 \right)^{\frac{N+\a}{2}}}\, dr d\s d\bar{s}\\
&= \int_{\R}\int_{S^{N-2}}2 \int_{u(\bar{s})}^{+\infty} \frac{r^{N-2}}{\left(  (s-\bar{s})^2+u^2(s)+r^2-2ru(s)\s_1 \right)^{\frac{N+\a}{2}}} \, dr d\s d\bar{s}\\
& \hspace{10mm}-
 \int_{\R}\int_{S^{N-2}}   \int_0^{+\infty}  \frac{r^{N-2}}{\left(  (s-\bar{s})^2+u^2(s)+r^2-2ru(s)\s_1 \right)^{\frac{N+\a}{2}}}\, dr d\s d\bar{s}.
\end{align*}
By making the change of variable $r=u(s)\t$, we get \eqref{eq:NMC-varphi}. 

We next find the simple expression for the linearized operator given in \eqref{eq:dHuv}. Taking $u \equiv \kappa \in \cO$ a constant function, using \eqref{eq:NMC-varphi}, calling $t=s-\bar{s}$,
and denoting the partial derivatives of $I$ by $I_q$ and $I_p$,
we have 
\begin{align*}
\frac{d }{d\e} H({\kappa+\e v})\Big|_{\e=0}(s)= & -(1+\a)\kappa^{-2-\a}v(s)\int_{\R}\left\{ 
2I\left(1,\frac{t}{\kappa}\right)-I\left(0,\frac{t}{\kappa}\right)  \right\}dt\\
&+ \kappa^{-2-\a}\int_{\R}
  2 I_q\left(1,\frac{t}{\kappa}\right) \left( v(s-t)-v(s) \right) dt\\
  &+ \kappa^{-2-\a}v(s)\int_{\R}\left\{ 2I_p\left(1,\frac{t}{\kappa}\right)
  - I_p\left(0,\frac{t}{\kappa}\right) \right\}\frac{-t}{\kappa} dt.
\end{align*}
Using that $I_p\left(\cdot,\frac{t}{\kappa}\right)\frac{1}{\kappa}=
\de_t I(\cdot,\frac{t}{\kappa})$ and integrating by parts the third line in the 
previous expression, we find
\begin{align*}
\frac{d }{d\e} H({\kappa+\e v})\Big|_{\e=0}(s)= & -\a \kappa^{-2-\a}v(s)\int_{\R}
\left\{ 2I\left(1,\frac{t}
{\kappa}\right)-I\left(0,\frac{t}{\kappa}\right)  \right\}dt\\
&+ \kappa^{-2-\a}\int_{\R}\
  2 I_q\left(1,\frac{t}{\kappa}\right) \left( v(s-t)-v(s) \right) dt\\
&- \kappa^{-2-\a}v(s) \left[ t 2  I\left(1,\frac{t}{\kappa} \right) 
  -  t I\left(0,\frac{t}{\kappa}\right)  \right]_{t=-\infty}^{t=+\infty}\\
  =& -\a \kappa^{-2-\a}v(s)\int_{\R}
\left\{ 2I\left(1,\frac{t}
{\kappa}\right)-I\left(0,\frac{t}{\kappa}\right)  \right\}dt\\
&+ \kappa^{-2-\a}\int_{\R}\
  2I_q\left(1,\frac{t}{\kappa}\right) \left( v(s-t)-v(s) \right) dt.
\end{align*}
Making the change of variable $\tau=\frac{t}{\kappa}$, we get 
\begin{align*}
\frac{d }{d\e} H({\kappa+\e v})\Big|_{\e=0}(s)= 
& -\a \kappa^{-1-\a}v(s)\int_{\R}
\left\{ 2I\left(1,\t\right)-I\left(0,\t\right)  \right\}d\t\\
&+ \kappa^{-1-\a}\int_{\R}\
  2I_q\left(1,\t\right) \left( v(s-\kappa\t)-v(s) \right) d\t\\
 =& \kappa^{-1-\a}\int_{\R}
 G_\alpha(\tau) \left( v(s)-v(s-\kappa\tau) \right) d\tau -\kappa^{-1-\a} b_\a v(s),
\end{align*} 
since, by \eqref{eq:NMC-varphi}, $\int_\R \{ 2I(1,\t)-I(0,\t)\} d\t=H(1)$ and on the other hand, by  \eqref{new-rep-H}, 
$\alpha H(1)= \int_{S^{N-2}}   \int_{\R}
|\sigma-e_1|^2\{\t^2+|\sigma-e_1|^2 \}^{-(N+\a)/2} d\t d\s=b_\alpha$. We have also used that
$$
2I_q(1,\t)= -2 \int_{S^{N-2}}  \frac{1}{(\t^2+2(1-\sigma_1) )^{(N+\alpha)/2}}\,
d\sigma =-G_\a(\t).
$$
Thus, we have obtained the expression \eqref{eq:dHuv} for the linearized operator.

\subsection{Preliminary estimates on the linearized operator}
The following lemma provides estimates for the function $G_\alpha$ appearing in Proposition~\ref{lem:diff-H}. 
\begin{lemma}\label{lem:G_1}
Let  $N\geq 3$ and $\a>0$. Then there exists a positive constant $C$ depending only on $N$ and $\a$ such that 
\be \label{eq:estG1-leq}
G_{\a}(\tau)\leq \, C \min \left\{  |\tau|^{-2-\a},|\tau|^{-N-\a} \right\} \qquad \text{for $\tau\neq 0$}.
\ee
Moreover, we have 
\begin{equation}
  \label{eq:representation-G}
G_{\a}(\tau) = |\tau|^{-2-\a}g(\tau^2) \qquad \text{for $\tau \not = 0$},
\end{equation}
where $g:(0,+\infty) \to \R$ is a bounded function and it is given by 
\be \label{eq:def-g}
g(\rho)= 2  C_N   
\int_0^{2/\rho}\frac{(t(2-\rho t))^{\frac{N-4}{2}}}{\left( 1+2 t \right)^{\frac{N+\a}{2}}}\, dt 
\qquad \text{for $\rho>0$,}
\ee 
with $C_N=\frac{2\pi^{(N-2)/2}}{\G((N-2)/2)}$. 

Furthermore, we have 
\begin{equation}
  \label{def-g-0}
g_0:=  C_N'  \int_0^{+\infty}\frac{t^{\frac{N-4}{2}}}{\left( 1+2 t \right)^{\frac{N+\a}{2}}}\, dt = \lim_{\rho \to 0^+} g(\rho)   
\end{equation}
with $C_N'=2^{\frac{N-2}{2}}C_N =  2^{\frac{N-2}{2}}  \frac{2\pi^{(N-2)/2}}{\G((N-2)/2)}$, and 
\be\label{eq:gr2-g0}
|g(\rho )-g_0|\leq C  \rho  \qquad \text{for $\rho\in (0,1) $}
\ee
for some constant $C>0$. 
\end{lemma}

\begin{proof}
In the following, the letter $C$ stands for different positive constants depending only on $N$ and $\alpha$. 
We have that 
\begin{align*}
G_{\a}(\tau)=2\int_{S^{N-2}}\frac{1}{\left( {\tau^ 2}+2(1-\s_1)  \right)^{\frac{N+\a}{2}}}\, d\s&=2 C_N 
\int_{-1}^1\frac{(1-\s_1^2)^{\frac{N-4}{2}}}{\left( {\tau^ 2}+2(1-\s_1)  \right)^{\frac{N+\a}{2}}}\, d\s_1\\
&=2 C_N 
\int_0^2\frac{(r(2-r))^{\frac{N-4}{2}}}{\left( {\tau^ 2}+2 r \right)^{\frac{N+\a}{2}}}\, dr,
\end{align*}
with $C_N=|S^{N-3}|=\frac{2\pi^{(N-2)/2}}{\G((N-2)/2)}$. This leads to  
\be\label{eq:est-upGa}
G_{\a}(\tau)\leq 2C_N |\tau|^{-N-\a} \int_0^2{(r(2-r))^{\frac{N-4}{2}}} \, dr\leq C |\tau|^{-N-\a} \qquad \text{for $\tau \not = 0$.} 
\ee
Making the further change of variable $t=r/\tau^2$, we also find that 
$$
G_{\a}(\tau)= \frac{2  C_N}{|\tau|^{N+\a}} \int_0^{2}\frac{(r(2- r))^{\frac{N-4}{2}}}{\left( 1+2 r/\tau^2 \right)^{\frac{N+\a}{2}}}\, dr\\
=\frac{2  C_N}{|\tau|^{2+\a}} 
\int_0^{2/\tau^2}\frac{(t(2-\tau^2 t))^{\frac{N-4}{2}}}{\left( 1+2 t \right)^{\frac{N+\a}{2}}}\, dt =  \frac{g(\tau^2)}{|\tau|^{2+\a}}
$$
for $\tau \not =0$, with $g$ defined in~\eqref{eq:def-g}. 

Next we prove \eqref{eq:gr2-g0}. For this we write 
\begin{align}
g_0-g(\rho)&= 2 C_N\int_{0}^{2/\rho} \frac{(2^{\frac{N-4}{2}}-(2-\rho t)^{\frac{N-4}{2}}) t^{\frac{N-4}{2}}}{\left( 1+2 t \right)^{\frac{N+\a}{2}}}\, dt+   2^{\frac{N-2}{2}} C_N\int_{2/\rho}^{+\infty} \frac{  t^{\frac{N-4}{2}}}{\left( 1+2 t \right)^{\frac{N+\a}{2}}}\, dt \nonumber\\
&=: 2C_N I_1(\rho)+ 2^{\frac{N-2}{2}} C_N I_2(\rho). \label{eq:def_I1-I2}
\end{align}
We start with $I_2$ and notice that 
$$
I_2(\rho)= \int_{2/\rho}^{+\infty} \frac{   t^{\frac{N-2}{2}} t^{-1}  }{\left( 1+2 t \right)^{\frac{N+\a}{2}}}\, dt\leq \frac{\rho}{2}
\int_{0}^{+\infty} \frac{   t^{\frac{N-2}{2}}  }{\left( 1+2 t \right)^{\frac{N+\a}{2}}}\, dt,
$$
which yields 
\be \label{eq:I2}
I_2(\rho)\leq C \rho \qquad \text{for $\rho>0$.} 
\ee

On the other hand, if $N=4$, we have $I_1(\rho)=0$ for $\rho \in (0,1)$, and thus \eqref{eq:gr2-g0} follows. We now consider the case $N\ge 3$, $N \not=4$, and we write $I_1(\rho)$ as follows:
\begin{align}
I_1(\rho) &=\int_{0}^{2/\rho} \frac{(2^{\frac{N-4}{2}}-(2-\rho t)^{\frac{N-4}{2}}) t^{\frac{N-4}{2}}}{\left( 1+2 t \right)^{\frac{N+\a}{2}}}\, dt \nonumber \\
&= \int_{0}^{2/\rho}  \int_0^{1} \frac{d}{d\varrho} \frac{(2^{\frac{N-4}{2}}-(2- \varrho \rho t)^{\frac{N-4}{2}}) t^{\frac{N-4}{2}}}{\left( 1+2 t \right)^{\frac{N+\a}{2}}}\, d\varrho\, dt \nonumber\\
& = \frac{N-4}{2} \int_{0}^{2/\rho} \int_{0}^1 \rho t \frac{(2-\varrho \rho  t)^{\frac{N-6}{2}} t^{\frac{N-4}{2}}}{\left( 1+2 t \right)^{\frac{N+\a}{2}}}  d\varrho\, dt \nonumber\\
&= \frac{N-4}{2}\rho\int_{0}^1 \Bigl( I_{11}(\rho,\varrho)+ I_{12}(\rho,\varrho)\Bigr)\,d \varrho \label{eq:I1}
\end{align}
with 
$$
I_{11}(\rho,\varrho):= \int_{0}^{1/\rho} \frac{(2- \varrho \rho t)^{\frac{N-6}{2}} t^{\frac{N-2}{2}}}{\left( 1+2 t \right)^{\frac{N+\a}{2}}} dt 
\quad \text{and}\quad I_{12}(\rho,\varrho) :=\int_{1/\rho}^{2/\rho} \frac{(2- \varrho \rho t)^{\frac{N-6}{2}} t^{\frac{N-2}{2}}}{\left( 1+2 t \right)^{\frac{N+\a}{2}}} dt. 
$$
To estimate $I_{11}$, we observe that $2-\varrho\leq 2-\varrho \rho t\leq 2$ if $0\leq t\leq \frac{1}{\rho}$.  Consequently,
\begin{equation}
I_{11}(\rho,\varrho) \leq  \max \left\{ (2-\varrho)^{\frac{N-6}{2}}, 2^{\frac{N-6}{2}}  \right\} \int_{0}^{+\infty} \frac{   t^{\frac{N-2}{2}}  }{\left( 1+2 t \right)^{\frac{N+\a}{2}}}\, dt  \le C \label{eq:I11}
\end{equation}
for $\rho,\varrho \in (0,1)$. Moreover,  for $\rho, \varrho \in (0,1)$, we have   
\begin{align*}
I_{12}(\rho,\varrho)&\leq C  \int_{1/\rho}^{2/\rho} {(2- \varrho \rho t)^{\frac{N-6}{2}} t^{-1 - \a/2}} \, dt 
=C (\varrho \rho)^{\a/2}\int_\varrho^{2\varrho} (2-s)^{\frac{N-6}{2}}s^{-1-\a/2}\, ds\\
&\le \frac{C}{\varrho} \int_\varrho^{2\varrho} (2-s)^{\frac{N-6}{2}}\, ds.
\end{align*} 
If $\varrho \in (0,\frac{1}{2}]$, it thus follows that 
$$
I_{12}(\rho,\varrho)\leq  C \max_{0 \le s \le 1} (2-s)^{\frac{N-6}{2}} \le C,
$$
whereas in case $\varrho \in [\frac{1}{2},1)$ we deduce, since $N\ge 3$, $N \not=4$,  that 
$$
I_{12}(\rho,\varrho)\leq  C | (2-2\varrho)^{\frac{N-4}{2}} -(2-\varrho)^{\frac{N-4}{2}}| \le C(1-\varrho)^{-\frac{1}{2}}.
$$

From the last two estimates and (\ref{eq:I11}), we infer that 
$$
\int_{0}^1(I_{11}(\rho,\varrho)+ I_{12}(\rho,\varrho))\, d\varrho\leq C \qquad \text{for $\rho \in (0,1)$,}
$$
and together with \eqref{eq:I1} this yields 
$$
I_1(\rho)\leq C \rho \qquad \text{for $\rho \in (0,1)$.}
$$
Combining this inequality with  \eqref{eq:def_I1-I2} and \eqref{eq:I2}, we conclude that   
$$
|g_0-g(\rho)|\leq C \rho   \qquad \text{for $\rho \in (0,1)$.}
$$
Therefore \eqref{eq:gr2-g0} follows, and \eqref{eq:gr2-g0} implies (\ref{def-g-0}). Moreover,  
from (\ref{def-g-0}) and \eqref{eq:est-upGa} we deduce that \eqref{eq:estG1-leq} holds.
Finally, using \eqref{def-g-0} and $g(\t^2)=|\t|^{2+\a}G_\a(\t)$ combined with \eqref{eq:est-upGa},
we see that $g$ is bounded on $(0,\infty)$, as claimed. 
\end{proof}

Our next result will be important to derive estimates for the eigenvalues of the operator 
$$
\vp\mapsto \int_{\R}(\vp(s)-\vp(s-\t)) G_\a(\t)\, d\t
$$
acting on even $2\pi$-periodic functions, see Lemma~\ref{lemmL0} in Section~\ref{sec:nonl-probl-solve}.  As we shall see in that lemma, the eigenvalues of this operator are expressed in terms of the function $h: [0,\infty) \to \R$ defined by 
\be \label{eq:def-h}
h(b):= \int_{\R}( 1-\cos(b \t))G_\a(\t)\, d\t=  2\int_{\R}\int_{S^{N-2}}\frac{1-\cos(b \t)}{\left( {\t^2}+2(1-\s_1)  \right)^{\frac{N+\a}{2}}}\, d\s d\t.
\ee
We note that $h$ is well defined by \eqref{eq:estG1-leq},
since $|1-\cos(b\tau)| = 2 \sin^2 \frac{b\tau}{2} \le \frac{b^2 \tau^2}{2}$ for $\tau \in \R$, $b \ge 0$.

\begin{lemma}\label{lem:eigen-monoton}
For $N\geq 3$, the function $h$ defined in \eqref{eq:def-h} is differentiable. Moreover, it satisfies $h(0)=0$, 
\begin{equation}
  \label{eq:h-derivative-pos}
  h'(b)=  \int_{\R}\t\sin(\t b) G_\a(\t)\, d\t  > 0 \qquad \text{for $b>0$}
\end{equation}
and 
\begin{equation}
  \label{eq:h-limit-b}
0< \lim_{b\to +\infty}\frac{h(b)}{b^{1+\a}} < +\infty.
\end{equation}
\end{lemma}
\begin{proof}
 We first note that $h(0)=0$ holds trivially by definition. Next we prove that $h$ is differentiable. Indeed using Lemma  \ref{lem:G_1}, for $0 \le b< b_0$ and $\t \in \R \setminus \{0\}$ we have  
\be \label{eq:estimate-differentiability-h}
|\t \sin(\t b) G_\a(\t)| \le C \min \{\,b_0 |\t|^2, |\t|\,\} |\t|^{-2-\a}=:e(\t), 
\ee
and the function $\t \mapsto e(\t)$ is integrable over $\R$. Hence a standard argument based on Lebesgue's theorem and the mean value theorem shows that the limit
\begin{equation}
  \label{eq:formula-derivative-h}
h'(b)= \lim_{\rho \to 0}  \int_{\R}  \frac{\cos(b \t)-\cos((b+\rho)\t)}{\rho} 
G_\a(\t)\,dt =  \int_{\R}\t\sin(\t b) G_\a(\t)\, d\t 
\end{equation}
exists. Hence $h$ is differentiable and satisfies the equality in (\ref{eq:h-derivative-pos}).

To prove (\ref{eq:h-limit-b}), using (\ref{eq:representation-G}) we see that 
\begin{equation}
  \label{eq:new-rep-h}
h(b)=\frac{1}{b} \int_{\R}(1-\cos( \t))G_\alpha(\t/b) d\t = b^{1+\alpha} \int_{\R} \frac{1-\cos( \t)}{ |\t|^{2+\alpha}} g(\t^2/b^2)\,d\t   
\end{equation}
for $b>0$, with $g$ given by \eqref{eq:def-g}. Since the function $g$ is bounded by Lemma~\ref{lem:G_1},  it follows from Lebesgue's theorem that 
$$
\lim_{b\to +\infty} \frac{h(b)}{b^{1+\a}} = g_0 \int_{\R}\frac{1-\cos(\t)}{|\t|^{2+\a}}\, d\t \quad \in \; (0,\infty)
$$
with $g_0$ defined in (\ref{def-g-0}). 

To prove that $h'(b)>0$ for $b>0$, we note that from \eqref{eq:def-g} we have,  for $\rho>0$,
$$
g'(\rho)= -  \, \frac{4C_N}{(1+\frac{4}{\rho})^{\frac{4+\alpha}{2}} \rho^2}< 0 \qquad \text{if $N=4$}
$$
and  
$$
g'(\rho)= - (N-4)C_N \int_0^{2/\rho} \frac{(2-\rho t)^{\frac{N-6}{2}} t^{\frac{N-2}{2}}}{\left( 1+2 t \right)^{\frac{N+\a}{2}}}\, dt<0 \qquad \text{if $N \ge 5$.}
$$
Thus in case $N \ge 4$ the function $b \mapsto g(\frac{\tau^2}{b^2})$ is increasing on $(0,\infty)$ for every 
$\tau \in \R \setminus \{0\}$. 
Consequently,  by (\ref{eq:new-rep-h}), the fact that $h$ is differentiable,
and the strict positivity of the function $g$, we immediately deduce that $h'(b)>0$ for $b>0$. 

In case $N=3$, we can use (\ref{eq:formula-derivative-h}) and the estimate (\ref{eq:estimate-differentiability-h}) to compute, by changing the order of integration,
\begin{align*}
h'(b)  &=\int_{\R}\t\sin(\t b) G_\a(\t)\, d\t = 2 \int_{\R}\t\sin(\t b) \int_{S^{1}}\frac{1}{(\tau^2+2(1-\s_1))^{\frac{3+\a}{2}}}\,d\s d\t \\
&= 2 \int_{S^{1}}  \int_{\R} \frac{\t\sin(\t b)}{(\tau^2+2(1-\s_1))^{\frac{3+\a}{2}}}\,  d\t d\s = 
\int_{S^{1}} V_b(\sqrt{2(1-\s_1)}) d\s,
\end{align*}
where 
\begin{equation}
\label{eq:Plambda1ms1}
V_b(\xi):=2\int_{\R} \frac{\t \sin(\t b)}{\left( \t^2+\xi^2  \right)^{\frac{3+\a}{2}}}\,  d\t=  4\int_{0}^\infty  
\frac{\t \sin(\t b)} {\left(  {\t^2} +\xi^2   \right)^{\frac{3+\a}{2}}}\,  d\t= 4\chi  b^{1+\a/2 } \xi^{-\alpha/2}  
K_{\a/2}(b\xi)
\end{equation}
for $\xi,b>0$. Here   $K_\nu$ is the modified Bessel function of the second kind (also called Macdonald function),  
$\chi:=  2^{-1-\alpha/2} \frac{\sqrt{\pi}}{\Gamma((3+\a)/2) }  >0$ 
and $\Gamma$ is the usual Gamma function. For the second equality in (\ref{eq:Plambda1ms1}), 
we refer e.g. to \cite[Page 442,  3.771, 5.]{GR} and note that $K_{\alpha/2}=K_{-\alpha/2}$
(see \cite[Page 929, 8.486, 16.]{GR}). 
Since $K_{\a/2}(b\xi)>0$ (by \cite[Page 917, 8.432, 1.]{GR}) and therefore $V_b(\xi)>0$ for $\xi,b>0$, 
it thus follows that $h'(b)>0$ for $b>0$, 
completing the proof of the lemma.
\end{proof}

\section{Nonlinear problem to be solved and Proof of Theorem~\ref{res:cyl1}}
\label{sec:nonl-probl-solve}

To prove Theorem~\ref{res:cyl1}, we are looking for constants $R, a_0>0$ and functions $u_a$ of the form
$$
u_a(s)= R+\frac{\varphi_a(\lambda(a) s)}{\l(a)} , \qquad  a \in (-a_0,a_0),   
$$
satisfying the equation
\be\label{eq:aim}
H(u_a)(s)= H(R)\qquad \textrm{ for all } s\in\R, \; a \in (-a_0,a_0).   
\ee
Here we require that $\lambda: (-a_0,a_0) \to (0,\infty)$ is a smooth function such that $\lambda(0)=1$.
Moreover, we look for functions 
$\varphi_a \in C^{1,\beta}(\R)$ with $a \in (-a_0,a_0)$ which are even, $2\pi$-periodic, and 
satisfy the expansion
$$
\varphi_a = a(\cos(\cdot)+v_a) 
$$
with $v_{a}\to 0$ in $C^{1,\beta}(\R)$ as $a\to 0$ and $\int_{-\pi}^{\pi}v_{a}(t)\cos(t)\, dt=0$ for $a \in (-a_0,a_0)$.

Note that we have rescaled the problem so that we can work with functions $\varphi_a$
with fixed period. For the rescaled function $\tilde u_a(s):=\l(a) u_a(\frac{s}{\l(a)})$, a change of variables gives 
$$
H(\tilde u_a)(s)=\l(a)^{-\a}H(u_a)\left( \frac{s}{\l(a)}\right)\qquad \text{for $s \in \R.$}
$$
Therefore by \eqref{eq:aim}  our problem becomes 
\begin{equation*}
\label{eq:aim2-scale}
H(\lambda(a) R + \varphi_a)(s)=  H(\tilde u_a)(s)= \l(a)^{-\a} H(R)=H(\l(a) R)\qquad \textrm{ for all } s\in\R. 
\end{equation*}
For matters of convenience, we will use $\mu(a)= \lambda(a) R$ as a new unknown.  
Our aim is to deduce Theorem \ref{res:cyl1} from the Crandall-Rabinowitz theorem~\cite{Crandall1971} applied to the map 
$$
  \label{eq:defPhi1}
(\mu, \varphi) \mapsto \Phi(\mu,\varphi):= \mu^{1+\a} \bigl \{H(\mu + \varphi) - H(\mu)\bigr\},  
$$
since our equation has become $\Phi(\mu,\varphi)=0$.
The factor $\mu^{1+\a}$ is introduced  to simplify some expressions at a later stage. 

We need to introduce the functional spaces in which we work. We fix $\beta$ such that
\begin{equation}\label{defbeta}
0<\alpha <\beta <\min\{1,2\alpha+1/2\}.
\end{equation}
The condition $\beta < 2\alpha+1/2$ is technical (to simplify a proof on regularity) and could
be avoided. Consider the Banach spaces
$$
 X:=C^{1,\beta}_{p,e}=\{\va:\R\to\R\, :\ \va\in C^{1,\beta}(\R) \text{ is $2\pi$-periodic and even}\}
$$
and
$$
 Y:=C^{0,\beta-\alpha}_{p,e}=\{\tilde\va:\R\to\R\, :\ \tilde\va\in C^{0,\beta-\alpha}(\R) 
 \text{ is $2\pi$-periodic and even}\}.
$$
The norms of $X$, respectively $Y$, are the standard $C^{1,\beta}(\R)$ and $C^{0,\beta-\alpha}(\R)$-norms, 
respectively, defined by
\begin{equation}
\label{eq:def-hoelder-norm}
\|u\|_{C^{k,\gamma}(\R)} : = \sum_{j=0}^k \|u^{(j)}\|_{L^\infty(\R)} +  \sup_{\stackrel{s,t \in \R}{s \not = t}}\frac{|u^{(k)}(s)-u^{(k)}(t)|}{|s-t|^{\gamma}}. 
\end{equation}

Since $H: \calO \subset C^{1,\beta}(\R) \to C^{0, \b-\a}(\R)$ is smooth in $\cO$ by Proposition \ref{lem:diff-H} (for $N\geq 3$), 
and clearly $H$ sends $2\pi$-periodic and even functions to functions which are also $2\pi$-periodic and even 
(for instance by expression \eqref{new-rep-H}), we infer that 
\be\label{defPhi}
\Phi: \cD_\Phi  \to Y,\qquad \Phi(\mu,\varphi) = \mu^{1+\a} \bigl \{H(\mu + \varphi) - H(\mu)\bigr\}
\ee
is a smooth map defined on the open set 
\begin{equation}
  \label{eq:DPhi}
\cD_\Phi:= \{(\mu,\varphi)\::\: \mu>0,\: \varphi \in X,\, \inf_{\R} \varphi>-\mu\} \;\subset \; \R \times X.
\end{equation}
By definition, we have 
$$
\Phi(\mu,0)=0 \qquad \text{for every $\mu>0$.}
$$
Next we need to study the properties of the family of linearized operators 
$$
L_\mu := D_\varphi \Phi(\mu,0) = \mu^{1+\alpha} DH(\mu) \;\in \; \cL(X,Y), \qquad \quad \mu>0.
$$
Here and in the following, $\cL(X,Y)$ denotes the space of bounded linear operators \mbox{$X \to Y$.} By Proposition~\ref{lem:diff-H}, $L_\mu$ is given by 
\begin{equation}
   \label{eq:def-L-operator-1}
  L_\mu v(s)= PV \int_{\R} (v(s)- v(s-\mu \tau))G_\a(\tau)\, d\tau -  b_\alpha  \,   v(s) \qquad \text{for $v \in X$.}
\end{equation}
 
\begin{lemma}\label{lemmL0}
\label{sec:nonl-probl-solve-1}
Let $N\geq 3$ and $\mu>0$. The functions 
\begin{equation}
  \label{eq:def-e-k}
e_k \in X, \quad e_k(s)=\cos(ks), \quad k \in \N \cup \{0\}
\end{equation}
are eigenfunctions of $L_\mu$ with corresponding  eigenvalues 
\be \label{eq:lkhk}
\l_k(\mu)=h(k \mu ) -b_\alpha, \qquad k \in \N\cup \{0\},
 \ee
where the function $h$ is defined in \eqref{eq:def-h}. Moreover, 
 \begin{align}
\label{eigen}
&\lambda_0(\mu) <\lambda_1(\mu) <\lambda_2(\mu) <\dots \qquad \text{and}\\ 
&0<  \lim_{k\to +\infty}\dfrac{\lambda_k(\mu)}{k^{1+\alpha}} <\infty .
\label{assym}
\end{align}
\end{lemma}

\begin{proof}
Let $k \in \N \cup \{0\}$. We note that, by (\ref{eq:def-L-operator-1}),
\begin{align*}
&L_\mu e_k(s)= PV \int_{\R} \left\{\cos(ks)-\cos(k(s-\mu \tau))\right\}G_\a(\tau)\, d\tau -  b_\alpha  \,   \cos(ks)\\
&=  PV  \int_{\R} \left\{\cos(ks)-\cos(ks) \cos(k \mu \tau )-\sin(ks) \sin(k \mu \tau ) \right\}G_\a(\tau) \, d\tau -   b_\alpha \,   \cos(ks)\\
&=   \left\{ PV \int_{\R}\left\{ 1-\cos(k \mu \tau )\right\}\, G_\a(\tau )\, d\tau-  b_\alpha \right\} \,   \cos(k s) = \bigl(h(k \mu) -b_\alpha\bigr)e_k(s),
\end{align*}
with $h$ defined in \eqref{eq:def-h}. Here we used the oddness of $\sin(\cdot)$ and the evenness of $G_\alpha$. This shows that $e_k$ is an eigenfunction of $L_\mu$ with eigenvalue $h(\mu k)-b_\alpha$. The properties~(\ref{eigen}) and (\ref{assym}) now follow readily from Lemma \ref{lem:eigen-monoton}.
\end{proof}

We are now in position to establish the following.

\begin{proposition}\label{propPhi}
Let $N\geq 3$. There exists a unique $\mu_*>0$ such that the linear operator $L:= L_{\mu_*}: X \to Y$ has the following properties. 
\begin{itemize}
\item[(i)] The kernel of $L$ is spanned by the function $\cos(\cdot)$.
\item[(ii)] The range of $L$ is given by 
$$
R(L)= \Bigl \{v \in Y \::\: \int_{-\pi}^{\pi} \cos(s) v(s)\,ds = 0 \Bigr\}. 
$$
\end{itemize}
Moreover, we have that 
\begin{equation}
  \label{eq:transversality-cond}
\partial_\mu \Bigl|_{\mu= \mu_*} L_\mu \cos(\cdot) \not  \in R(L).   
\end{equation}
\end{proposition}

\begin{proof}
By Lemmas~\ref{lem:eigen-monoton} and~\ref{lemmL0}, there exists a unique $\mu_*>0$ such that \mbox{$L_{\mu_*}\cos(\cdot) = 0$.} We put $L:= L_{\mu_*}$ in the following. Consider the spaces 
\begin{align}
X_\perp &:= \Bigl \{v \in X \::\: \int_{-\pi}^{\pi} \cos(s) v(s)\,ds = 0 \Bigr\} \subset X, \label{def-X-perp}\\
Y_\perp &:= \Bigl \{v \in Y \::\: \int_{-\pi}^{\pi} \cos(s) v(s)\,ds = 0 \Bigr\} \subset Y. \nonumber
\end{align}
To show properties (i) and (ii), it clearly suffices to prove that 
\begin{equation}
  \label{eq:isomorphism}
\text{$L$ defines an isomorphism between $X_\perp$ and $Y_\perp$.}  
\end{equation}
To prove (\ref{eq:isomorphism}), we let 
\begin{align}
H_\perp &:= \Bigl \{v \in H^{1+\alpha}_{loc}(\R) \::\: \text{$v$ even, $2\pi$-periodic with } 
\int_{-\pi}^{\pi} \cos(s) v(s)\,ds = 0 \Bigr\}, \label{defH*}\\
V_\perp &:= \Bigl \{v \in L^2_{loc}(\R) \::\: \text{$v$ even, $2\pi$-periodic with } 
\int_{-\pi}^{\pi} \cos(s) v(s)\,ds = 0 \Bigr\}\nonumber.
\end{align}
We note that the functions $\cos(k \cdot)$, $k \in \{0,2,3,4,\dots \}$ form an orthonormal basis of $V_\perp$, and that $H_\perp$ can be characterized in terms of Fourier coefficients as the subspace of all $v \in V_\perp$ such that
$$
\sum_{k \in \N} \Bigl(k^{1+\alpha} \int_{-\pi}^{\pi} v(s) \cos(ks)\,ds\Bigr)^2 < \infty.
$$
Since $\cos(k\cdot)$ are eigenfunctions of $L$ with eigenvalues $\lambda_k(\mu_*)$, 
from \eqref{eigen}, the fact that $L\cos(\cdot) =0$, the asymptotics \eqref{assym}, 
and the characterization given above we deduce that 
\begin{equation}
\label{eq:Sobolev-invertible}
\text{$L$ defines an isomorphism between $H_\perp$ and $V_\perp$}.
\end{equation}
Next, note that $C^{1,\beta}(\R)\subset H^{1+\alpha}_{loc}(\R)$. This follows from the definition of $H^{1+\alpha}_{loc}(\R)$ via the Gagliardo seminorm, see e.g. \cite[Definition 1.3.2.1]{grisvard}. Indeed, let $v\in C^{1,\beta}(\R)$. 
To see that $v'\in H^{\alpha}_{loc}(\R)$ we need to ensure that 
$$
\int_{\Omega}\int_{\Omega} |v'(s)-v'(\bar{s})|^2 |s-\bar{s}|^{-(1+2\alpha)}\,dsd\bar{s} <
\infty
$$
for any bounded interval $\Omega \subset \R$. This is clearly true since $v'\in C^{0,\beta}(\R)$ and $\beta>\alpha$.

We deduce that $X_\perp \subset H_\perp$. Since also $Y_\perp = V_\perp \cap Y$, we see that $L: X_\perp \to Y_\perp$ is well 
defined and one-to-one. 

To establish surjectivity, let $f \in Y_\perp$. Since $Y_\perp\subset V_\perp$, 
by (\ref{eq:Sobolev-invertible}) there exists $w\in H_\perp$ such that $Lw=f$. Recall that, by (\ref{eq:def-L-operator-1}) and a change of variable, $L$ is given by 
$$
Lw(s)=   \int_{\R} \left\{w(s)-w(s-t)\right\}G_\a\left({t}/{\mu_*}\right)\, \frac{dt}{\mu_*} -  b_\alpha  \,   w(s).
$$
  Hence, $Lw=f$ can be written as
$$
\int_{\R} \left\{w(s)-w(s-t)\right\}G_\a\left({t}/{\mu_*}\right)\, dt = \mu_* b_\alpha w(s) + \mu_* f(s) \qquad\text{for $s \in \R$.}
$$
Moreover, $w\in H_\perp \subset Y=C^{0,\beta-\alpha}_{p,e}$ by Morrey's embedding, 
since $1+\alpha-1/2= 1/2+\alpha>\beta-\alpha$ as assumed in \eqref{defbeta}. 
Thus $\mu_* b_\alpha w + \mu_* f\in Y$, and  Lemma \ref{lem:Regularity} below yields $w\in X \cap H_\perp= X_\perp$. 
The proof of (\ref{eq:isomorphism}) is complete.

It remains to prove~\eqref{eq:transversality-cond}, which is simply a consequence of the fact that 
$$
\partial_\mu \Bigl|_{\mu= \mu_*} L_\mu \cos(\cdot) =  \partial_\mu \Bigl|_{\mu= \mu_*}\l_1(\mu) \cos(\cdot)= h'(\mu_*) \cos(\cdot) 
$$
by \eqref{eq:lkhk} and that $h'(\mu_*)>0$ by Lemma~\ref{lem:eigen-monoton}.
\end{proof}

It remains to prove the regularity result that we have used at the end of the previous proof.

\begin{lemma}\label{lem:Regularity}
Let $N\geq 3$, $\mu>0$, $f\in Y$ and $v\in H_\perp \subset Y$ be such that 
\be\label{eq:duG1eqg} 
\int_{\R} (v(s)-v(s-t))G_{\a}\left({t/\mu }\right)\, dt =f(s) \qquad \textrm{  for all $s\in \R$}, 
\ee 
where $H_\perp$ is defined in (\ref{defH*}). Then $v\in X=C^{1,\beta}_{p,e}$.
\end{lemma}

\begin{proof}
Put $\Gamma_v(s,t)=v(s)-v(s-t)$ for $s,t \in \R$. Recalling  Lemma \ref{lem:G_1}, we write 
\begin{align}
f(s) = \int_{\R}  \G_v(s,t)G_{\a}\left({t/\mu}\right)\, dt &=\mu^{2+\alpha} \int_{\R} \G_v(s,t) |t|^{-2-\a}g(t^2/\mu^2)\, dt\nonumber \\
&=\mu^{2+\alpha} g_0 \int_{\R} \G_v(s,t) |t|^{-2-\a}\, dt+ \ti{f}(s), \label{eq:form-Ga-to-Frac}
\end{align}
where
$$
\ti{f}: \R \to \R, \qquad \ti{f}(s):= \mu^{2+\alpha} \int_{\R} \G_v(s,t) |t|^{-2-\a}\bigl(g(t^2/\mu^2)-g_0\bigr)\, dt.
$$
By Lemma \ref{lem:G_1}, $g$ is a bounded function on $(0,\infty)$ which satisfies $|g(t^2/\mu^2)-g_0| \le C t^2$ for $t \in (-\mu,\mu)$ by \eqref{eq:gr2-g0}. Moreover, we have $| \G_v(s,t)- \G_v(\ov{s},t)|\leq 2\|v\|_{Y}|s-\ov{s}|^{\beta-\alpha}$ for $s, \ov{s},t \in \R$. Consequently, we deduce that  
$$
\|\ti{f}\|_{Y}\leq C \|v\|_{Y}. 
$$
with a constant $C>0$ independent of $v$.  Now thanks to \eqref{eq:form-Ga-to-Frac}, the equality \eqref{eq:duG1eqg} becomes
$$
 \int_{\R} \frac{v(s)-v(s-t)}{ |t|^{2+\a}}\, dt=\frac{1}{\mu^{2+\alpha}g_0} \bigl(f(s) -\ti{f}(s)\bigr)  \qquad \textrm{ for every   } s\in \R.
$$
Since $f-\ti{f}\in Y\subset C^{0,\b-\a}(\R)$ and $v\in Y\subset L^\infty(\R)$, by  
\cite[Proposition~2.8]{Sil}  we conclude  that $v\in X$. 
\end{proof}

We are now in a position to apply the Crandall-Rabinowitz theorem \cite{Crandall1971}, which will give rise to the following bifurcation property. 
\begin{proposition}\label{propPhi-crandall-rabinowitz}
\label{sec:nonl-probl-solve-2}
For $N\geq 3$, let $\mu_*$ be defined as in Proposition~\ref{propPhi}, let $X_\perp \subset X$ be the closed subspace given in (\ref{def-X-perp}), so that $X= X_\perp \oplus \langle  \cos(\cdot) \rangle$. Moreover, let $\cD_\Phi \subset \R \times X$ 
be the open set defined in (\ref{eq:DPhi}). Then there exists $a_0>0$ and a $C^\infty$ curve 
$$
(-a_0,a_0) \to \cD_\Phi, \qquad a \mapsto (\mu(a),\varphi_a) 
$$
such that 
\begin{itemize}
\item[(i)] $\Phi(\mu(a),\varphi_a)=0$ for $a \in (-a_0,a_0)$. 
\item[(ii)] $\mu(0)= \mu_*$. 
\item[(iii)]  $\varphi_a = a \bigl(\cos(\cdot) + v_a\bigr)$ for $a \in (-a_0,a_0)$, and 
$$
(-a_0,a_0) \to X_\perp, \qquad a \mapsto v_a 
$$
is a $C^\infty$ curve satisfying $v_0=0$.
\item[(iv)] $\mu(a)=\mu(-a)$ and 
$\varphi_{-a}(s)= \varphi_a(s+\pi)$ for $a \in (-a_0,a_0)$, $s \in \R$.
\end{itemize}
\end{proposition}

\begin{proof}
The claims (i)-(iii) follow by a direct application of the Crandall-Rabinowitz theorem as given in Theorems 1.7 and 1.18 of
\cite{Crandall1971}. The assumptions of this theorem are satisfied by Proposition~\ref{propPhi}. To see (iv), we put 
$$
\psi_{a}(s):=\varphi_a(s+\pi)= -a \bigl(\cos(s) + w_a\bigr) \qquad \text{for $a \in (-a_0,a_0)$,}
$$ 
where $w_a \in X$ is defined by $w_a(s):= - v_a(s+\pi)$. We then have 
$$
H(\mu(a)+\psi_a)(s)=H(\mu(a)+\phi_a)(s+\pi)=H(\mu(a))\qquad \text{for $s \in \R$, $a \in (-a_0,a_0)$}
$$
and thus $\Phi(\mu(a),\psi_a)=0$ for $a \in (-a_0,a_0)$. By the local uniqueness statement (1.8) in \cite[Theorem 1.7]{Crandall1971}, there exists $\eps \in (0,a_0)$ such that 
\begin{equation}
  \label{eq:set-equality}
\Bigl\{\Bigl(\mu(a), -a \bigl(\cos(\cdot) + w_a\bigr)\Bigr)\,:\, |a|<\eps\Bigr \} \subset 
\Bigl\{\Bigl(\mu(a),a \bigl(\cos(\cdot) + v_a\bigr)\Bigr)\,\:\, |a|< a_0 \Bigr\}
\end{equation}
By noting in addition that $w_a \in X_\perp$ as a consequence of the fact that $v_a \in X_\perp$, it follows that 
$$
\mu(a)=\mu(-a)\quad \text{and}\quad w_a=v_{-a}\qquad \text{for $a \in (-\eps,\eps)$,}
$$
hence also $\phi_{-a}= \psi_a$ for $a \in (-\eps,\eps)$. Replacing $a_0$ by $\eps$, we thus conclude that properties (i)-(iv) hold.
\end{proof}

\begin{remark}
{\rm
As in our 2D paper \cite{Cabre2015A}, one could avoid using the Crandall-Rabinowitz theorem by considering the map
$(a,\varphi) \mapsto\frac{1}{a}\{H(\mu+a\varphi)-H(\mu)\}$ instead of the map \eqref{defPhi}. In this way one uses
the implicit function theorem at $a=0$.

At the same time, we could have proved the 2D result in \cite{Cabre2015A} using the Crandall-Rabinowitz theorem as in
the present paper.
}
\end{remark}

\begin{proof}[Proof of Theorem~\ref{res:cyl1} (completed)]
Let $\mu_*$ be given by Proposition~\ref{propPhi}, and consider $a_0>0$ and the smooth curve 
$$
(-a_0,a_0) \to \cD_\Phi, \qquad a \mapsto (\mu(a),\varphi_a) 
$$ 
given by Proposition~\ref{propPhi-crandall-rabinowitz}. We put $R:= \mu_*$ and consider the smooth maps 
\begin{align*}
  &(-a_0,a_0) \to (0,\infty), \qquad \quad a \mapsto \lambda(a):= \frac{\mu(a)}{R}=\frac{\mu(a)}{\mu_*}\\ 
&(-a_0,a_0) \to C^{1,\beta}(\R), \qquad\; a \mapsto u_a = R + \frac{\varphi_a(\lambda(a) \,\cdot\,)}{\lambda(a)}.  
\end{align*}
With these definitions, all but two properties stated in Theorem~\ref{res:cyl1} follow immediately from Proposition~\ref{propPhi-crandall-rabinowitz} and the remarks at the beginning of this section -- note in particular that $\partial_a u_a \big|_{a=0}  = \cos(\cdot)$ follows from (1.4) and the fact that
$\lambda(0)=1$ and $v_0=0$. The following two statements still need to be justified:

{\em Claim I.} The minimal period of $u_a$ is
$2\pi/\lambda(a)$ if $a\not=0$. Clearly this is equivalent, after the rescaling, to the statement that the function 
$$
u(s)=\lambda R+ a\{\cos(s)+v_a(s)\},
$$
with $a\not=0$ and $v_a$ orthogonal to $\cos(\cdot)$ in $L^2(-\pi,\pi)$, has minimal period $2\pi$. 
This is easily proved by
expressing $v_a(s)$ as a Fourier series $a_0+\sum_{k=2}^{\infty}a_k\cos(ks)$. If $T$ is the minimal
period of $u$, we must have 
\begin{eqnarray*}
&& \hspace{-1cm}\cos(s)+v_a(s)=\cos(s+T)+v_a(s+T)\\
 &=&\cos(s)\cos(T)-\sin(s)\sin(T)
+a_0+\sum_{k=2}^{\infty}a_k\{\cos(ks)\cos(kT)-\sin(ks)\sin(kT)\}.
\end{eqnarray*}
Multiplying the first and last expressions in the above equalities by $\cos(s)$
and integrating in $(-\pi,\pi)$, we deduce that $\cos(T)=1$. Hence the minimal period is $T=2\pi$.

{\em Claim II.} We have $u_a\not\equiv u_{a'}$ if $a\not=a'$. Indeed,
if $u_a\equiv u_{a'}$, then the minimal periods of these functions coincide, and thus $\lambda(a)=\lambda(a')$.
By (\ref{form}) we then have
\begin{equation}
  \label{eq:final-explanation-identical}
R+\frac{a}{\lambda(a)}\{\cos(\cdot)+v_a(\cdot)\} = u_a(\frac{\cdot}{\lambda(a)}) = u_{a'}(\frac{\cdot}{\lambda(a)}) =
R+\frac{a'}{\lambda(a)}\{\cos(\cdot)+v_{a'}(\cdot)\},
\end{equation}
where the functions $v_a$ and $v_{a'}$ are orthogonal to
$\cos(\cdot)$ in $L^2(-\pi,\pi)$. Multiplying (\ref{eq:final-explanation-identical}) with $\cos(\cdot)$ and integrating over $[-\pi,\pi]$, we obtain 
$$
\frac{a}{\lambda(a)} \int_{-\pi}^{\pi} \cos^2(t)\,dt = \frac{a'}{\lambda(a)} \int_{-\pi}^{\pi} \cos^2(t)\,dt 
$$
and therefore $a=a'$.
\end{proof}

\section{Regularity of the NMC operator} \label{sec:preparations-1}
The purpose of this section is to give the proof of Proposition~\ref{lem:diff-H}. 
We first observe that obviously it suffices to consider $\d>0$ and to prove the regularity of the NMC operator as a map
$$
H: \cO_\d \to C^{0, \b-\a}(\R), \qquad \text{where $\cO_\d:= \{u \in C^{1,\beta}(\R)\::\: \inf_{\R} u > \delta\}$.}
$$
To accomplish this, it will be crucial to use the expression of  
$H$ given in Lemma~\ref{Hreg}.

For the readers convenience, let us first recall some notation introduced already in Subsection 2.1. We denote
$$ 
p_\s:=|\s-e_1|
$$
and, for $r\in\R$, we define 
$$
\mu_r(\s)= \frac{1}{|\s-e_1|^{N+r}} = p_\s^{-N-r}.
$$
It is easy to see that 
\be \label{eq:int-mu_n}
\int_{S^{N-2}}\mu_r(\s) \, d\s <\infty \qquad \textrm{ for every } r<-2 .
\ee

We define the  maps $\L_0, \L: C^{1,\b}(\R)\times\R\times\R\times \R\to \R $ by   
$$
\L_0(\phi,s,t,p) = \frac{\phi(s)-\phi(s-pt)}{pt} = \int_0^1 \phi'(s-\rho p t) d\rho 
$$
and 
$$
\L(\phi, s,t, p)= \L_0(\phi,s,t,p)-  \phi'(s-pt) =  \int_0^1(\phi'(s-\rho p t)-\phi'(s-p t) ) d\rho.
$$
We observe that  for every $s,s_1,s_2,t,p\in\R$, we have
\be\label{eq:Phis}
| \L(\phi, s,t, p)|\leq 2  \|\phi\|_{C^{1,\b}(\R)}\, \min( |t|^\b   |p|^\b ,1)
\ee
and also 
\be\label{eq:Phis1s2}
| \L(\phi, s_1,t, p)-\L(\phi, s_2,t, p)|\leq  2 \|\phi\|_{C^{1,\b}(\R)}\, \min( |t|^\b   |p|^\b, |s_1-s_2|^\b).
\ee
Note also that for every $s,s_1,s_2,t,p\in \R$, we have 
 \be\label{eq:est-denom-cK-s1-s2}
\left| \L_0(u,s_1,t,p)^2 -  \L_0(u,s_2,t,p)^2   \right| \leq 2 \|u\|_{C^{1,\b}(\R)}^2|s_1-s_2|^\b.
\ee

In Lemma~\ref{Hreg} we established that, for $u\in\mathcal{O}$, we have
\begin{align}\label{eq:good-expression2-NMC}
\cH(u)(s):=& -\frac{\a}{2} H({u})(s)\nonumber \\
=&\int_{S^{N-2}}\mu_{\a-2}(\s) \int_{\R}  t\L ({u}, s,t, p_\s)\cK_\a({u}, s,t,p_\s){u}^{N-2}(s-p_\s t) dt  d\s  \\
&\hspace{5mm}- \frac{{u} (s)}{2}    \int_{S^{N-2}}\mu_{\a-3}(\s) \int_{\R}  \cK_\a({u},s,t,p_\s){u}^{N-2}(s-p_\s t) dt  d\s \nonumber, 
\end{align}
where the function  $\cK_\alpha: C^{1,\b}(\R)\times\R\times\R\times\R\to \R$ is defined by 
\begin{equation}
\label{eq:defcKvarrho}
\cK_\alpha({u},s,t,p)=\frac{1}{\left( \displaystyle t^2+ t^2 \L_0(u,s,t,p)^2+{u}(s){u}(  s-p t)  \right)^{(N+\alpha)/2}}.
\end{equation}

Using this expression \eqref{eq:good-expression2-NMC} for the NMC, we shall show that $H: \cO_\d\to C^{0,\b-\a}(\R)$ is of class $C^\infty$ for every $\d>0$.
\subsection{Differential calculus toolbox}\label{ss:diff-cal}

For a finite set $\cN$, we let $|\cN|$ denote the length (cardinal) of $\cN$. It will be understood that $ |\emptyset|=0$. 
Let $Z$ be a  Banach space and $U$ a nonempty open subset of $Z$. If $T \in C^{k}(U,\R)$ and $u \in U$, then $D^kT(u)$ is a continuous 
symmetric $k$-linear form on $Z$ whose norm is given by  
$$
   \|D^{k}T ({u}) \|= \sup_{{u}_{1}, \dots,  {u}_{k}\in Z }
     \frac{|D^{k} T ({u})[u_1,\dots,u_k]| }{   \prod_{j=1}^k \|   {u}_{j} \|_{ Z }}    .
$$
If  $T_1,\, T_2 \in C^k(U,\R)$,  then also $T_1 T_2 \in C^k(U,\R)$, and the $k$-th derivative of $T_1 T_2$ at $u$ is given by 
\be \label{eq:Dk-T1T2}
D^k(T_1 T_2 )({u})[u_1,\dots,u_k]= \sum_{\cN \in  \scrS_k} D^{|\cN|} T_1({u})[u_n]_{n\in \cN} \,  D^{k-|\cN|} T_2({u}) [u_n]_{n\in \cN^c} ,
\ee
where  $\scrS_k $ is the set of subsets of $\{1,\dots, k\} $ and  $ \cN^c= \{1,\dots, k\}\setminus \cN $ for $\cN \in \scrS_k$.  If, in particular, $L: Z\to \R$ is a linear map, we have 
\be \label{eq:Dk-LT2}
D^{|\cN|}(L T_2 )({u})[u_i]_{i\in\cN}= L({u})    D^{|\cN|} T_2({u})[u_i]_{i\in\cN} +  
\sum_{j\in\cN} L({u}_j)   D^{|\cN|-1} T_2({u})  [u_i]_{\stackrel{i \in \cN}{ i\neq j}}.
\ee

We also recall the \textit{Fa\'{a} de Bruno formula}. We let $T$ be as above, $V \subset \R$ open with $T(U) \subset V$ and $g:V  \to \R$ be a  $k$-times differentiable map.  The Fa\'{a} de Bruno formula states that 
\be 
\label{eq:Faa-de-Bruno}
D^k( g\circ T)(u)[u_1,\dots,u_k]= \sum_{\Pi\in\scrP_k} g^{ (\left|\Pi\right|)}(T(u)) \prod_{P\in\Pi} D^{\left|P\right| }T(u)[u_j]_{j \in P} ,
 \ee
for $u, u_1,\dots,u_k  \in U$, where $\scrP_k$ denotes the set of all partitions of  $\left\{ 1,\dots, k \right\}$, see e.g. \cite{FaadeBruno-JW}.

\subsection{Regularity of the nonlocal mean curvature operator}\label{ss:regul-H}
For a function $u: \R \to \R$, we use the notation
$$
[u; s_1,s_2]:= u(s_1)-u(s_2)\qquad \text{for $s_1,s_2 \in \R$,} 
$$
and we note the obvious equality 
\be\label{eq:uv-s_1s_2} 
[uv; s_1,s_2] = [u;s_1,s_2]v(s_1) + u(s_2)[v;s_1,s_2] \qquad \text{for $u,v: \R \to \R$, $s_1,s_2 \in \R$.}
\ee 
We first give some estimates related to the kernel ${\cK}_\alpha$ as given in \eqref{eq:defcKvarrho}.
\begin{lemma}\label{lem:est-cK}
Let $N\geq 3$ and $k \in \N \cup \{0\}$. Then, there exists a constant $c=c(N,\alpha,\b,k,\d)>1$ such that  for all $(s,s_1,s_2,t,p)\in\R^5$ and ${u}\in \cO_\d$, we have 
   \be \label{eq:Dk-K-s}
\|  D_{u}^k {\cK}_\alpha({u},s,t ,p   )   \|\leq   
\frac{c(1+ \|{u}\|_{C^{1,\b}(\R)} )^{c}   }{    (1 + t^2)^{(N+\alpha)/2}    }  ,
 \ee 
   \be  \label{eq:Dk-K-s_1s_2}
\| [D_{u}^k {\cK}_\alpha({u},\cdot ,t ,p);s_1,s_2] \|\leq   
\frac{c(1+ \|{u}\|_{C^{1,\b}(\R)} )^{c}   \, |s_1-s_2|^\b }{    (1 + t^2)^{(N+\alpha)/2}}.     
 \ee
 \end{lemma}
\begin{proof}
Throughout this proof, the letter $c$ stands for different constants greater than one and depending only on $N,\alpha,\b,k$ and $\d$.
We define 
$$
Q:  C^{1,\b}(\R)\times \R\times \R\times \R\to \R,\qquad Q({u},s,t,p)=t^2 +t^2 \L_0(u,s,t,p)^2   + {u}(s){u}(s-p t)
$$ 
and
$$
g_\alpha \in C^\infty(\R_+, \R), \qquad 
g_\alpha(x)= x^{-(N+\alpha)/2},
$$
so that 
$$
{\cK}_\alpha({u},s,t,p)=g_\alpha\left(  Q({u},s,t,p)  \right) .
$$
 By \eqref{eq:Faa-de-Bruno} and recalling that $Q$  is quadratic in ${u}$,  we have 
 \begin{align}
 D_{u}^k {\cK}_\alpha({u},s,t,p)&[u_1,\dots,u_k] \nonumber\\
= &\sum_{\Pi\in\scrP_k^2}  g_\alpha^{(\left|\Pi\right| )  } ( Q({u},s,t,p)   ) \prod_{P\in\Pi} D^{\left|P\right| }_{u} Q({u},s,t,p)[u_j]_{j \in P},   \label{eq:Dk-K-s_1s_2-0}
 \end{align}
 where $\scrP_k^2$ denotes the set of partitions $\Pi$ of $\{1,\dots,k\}$ such that $|P| \le 2$ for every $P \in \Pi$. Hence by \eqref{eq:uv-s_1s_2} we have  
 \begin{align}
 \label{eq:Dk-K-s_1s_2-1}
& \hspace{-8mm}\left[ D_{u}^k {\cK}_\alpha({u}, \cdot,t ,p   )[u_1,\dots,u_k] ;s_1,s_2\right]\\
 =& \sum_{\Pi\in\scrP_k^2}  \left[  g_\alpha^{(\left|\Pi\right| )  }  ( Q({u},  \cdot,t,p)   ); s_1,s_2 \right] 
 \prod_{P\in\Pi} D^{\left|P\right| }_{u} Q({u},s_1,t,p)[u_j]_{j \in P}\nonumber\\
 &+\sum_{\Pi\in\scrP_k^2}    g_\alpha^{(\left|\Pi\right| )  }  ( Q({u},s_2,t,p)   )
 \Bigl[     \prod_{P\in\Pi} D^{\left|P\right| }_{u} Q({u},  \cdot,t,p)[u_j]_{j \in P}  \:;\:s_1,s_2 \Bigr].\nonumber
 \end{align}

 For $P \in \Pi$ with $|P| \le 2$,  by using \eqref{eq:est-denom-cK-s1-s2} and \eqref{eq:uv-s_1s_2}, we find that
 \begin{align}\label{eq:DP-Q-s_1s_2}
|[D^{|P|}_{u} &Q ({u},\cdot,t,p)[u_j]_{j \in P};s_1,s_2]|  \nonumber \\
&\leq  c(1+ \|{u}\|_{C^{1,\b}(\R)}^{2} )(1+t^2) |s_1-s_2|^\b \prod_{j \in P} \|{u}_j\|_{C^{1, \b}(\R)}  
\end{align}  
and 
 \begin{align}\label{eq:DP-Q-s}
 |D^{|P|}_{u} Q ({u},s  ,t,p)[u_j]_{j\in P}|
\leq  c(1+ \|{u}\|_{C^{1,\b}(\R)}^{2} )(1+t^2) \prod_{j \in P} \|{u}_j\|_{C^{1,\b}(\R)}. 
\end{align}
For $\ell \in \N$ and $x>0$, we have  
$$
  g^{(\ell)}_\alpha(x)=(-1)^{\ell}2^{-\ell} \prod_{i=0}^{\ell-1 } (N+\alpha +2 i) x^{-\frac{N+\alpha+2\ell}{2}}.
$$
Consequently, for every ${u}\in \cO_\d$, using \eqref{eq:est-denom-cK-s1-s2} and \eqref{eq:uv-s_1s_2}, we have the estimates
\begin{align}
&| \left[ g^{({\ell})}_\alpha  ( Q({u},\cdot, t,p   ) ) ;s_1,s_2\right]| \nonumber \\
&= \left |[ Q({u},\cdot  ,t,p   ) ;s_1,s_2]  \int_0^1  g^{({\ell}+1)}_\alpha  (\tau Q({u},s_1, t,p)+(1-\tau)Q({u},s_2, t,p))  d\tau \right|
   \nonumber\\
&\leq  c(1 + \|{u}\|_{C^{1,\b}(\R)}^2)\,    \, |s_1-s_2|^\b (1+t^2) (t^2+ \d^2)^{-\frac{N+\a+2{\ell} + 2}{2}}  \nonumber\\ 
&\leq  c(1 + \|{u}\|_{C^{1,\b}(\R)})^c  \frac{ |s_1-s_2|^\b}{ (1+t^2)^{\frac{N+\a+2{\ell}}{2}}} \label{eq:gQ-s1-s2}
\end{align}   
and 
\be  
\label{eq:gQ-s}
|   g^{({\ell})  }_\alpha  ( Q({u} , \cdot,t,p    )) |\leq  \frac{c}{    (1 +t^2)^{(N+\alpha+2 {\ell} )/2}    }  
\ee
for ${\ell}=0,\dots,k$. 

Therefore by  \eqref{eq:Dk-K-s_1s_2-1}, \eqref{eq:DP-Q-s_1s_2}, \eqref{eq:DP-Q-s}, \eqref{eq:gQ-s1-s2} and \eqref{eq:gQ-s}, 
we obtain
 \begin{align*}
&| \left[ D_{u}^k {\cK}_\alpha({u}, \cdot,t  ,p  )[u_1,\dots,u_k] ;s_1,s_2\right]|\\
 &\leq    c(1+ \|{u}\|_{C^{1,\b}(\R)} )^{c}|s_1-s_2|^\b  \sum_{\Pi\in\scrP_k^2}       \frac{1}{    (1 + t^2)^{\frac{N+\alpha+2 \left|\Pi\right|}{2}}    }
 \prod_{P \in \Pi} (1+t^2) \prod_{j \in P} \|{u}_j\|_{C^{1,\b}(\R)} \\
 &=      \frac{c(1+ \|{u}\|_{C^{1,\b}(\R)} )^{c}   \, |s_1-s_2|^\b }{    (1+ t^2)^{(N+\alpha)/2}    }   \sum_{\Pi\in\scrP_k^2  }    \prod_{P\in\Pi}  \prod_{j \in P} \|{u}_j\|_{C^{1, \b}(\R)}  .
 \end{align*} 
We then conclude that 
$$
| \left[ D_{u}^k {\cK}_\alpha({u}, \cdot,t ,p   )[u_1,\dots,u_k] ;s_1,s_2\right]|\leq   
\frac{c(1+ \|{u}\|_{C^{1,\b}(\R)} )^{c}   \, |s_1-s_2|^\b }{    (1+ t^2)^{(N+\alpha)/2}    }     \prod_{i=1}^k \|{u}_i\|_{C^{1, \b}(\R)}.
$$
This yields \eqref{eq:Dk-K-s_1s_2}.
 Furthermore we easily deduce from (\ref{eq:Dk-K-s_1s_2-0}), \eqref{eq:DP-Q-s} and \eqref{eq:gQ-s}    that 
$$
|  D_{u}^k {\cK}_\alpha({u},s,t ,p   )[u_1,\dots,u_k]   |\leq   
\frac{c(1+ \|{u}\|_{C^{1,\b}(\R)} )^{c}   }{    (1+ t^2)^{(N+\alpha)/2}    }     \prod_{i=1}^k \|{u}_i\|_{C^{1, \b}(\R)},
$$
completing the proof.
\end{proof}
The following two lemmas provide the desired estimates for the formal candidates to be the derivatives of $H$.
\begin{lemma}\label{lem:est-cand-deriv}  
Let $N\geq 3$, $\d>0$, ${u} \in \cO_\d$ and $\phi, {u}_1,\dots, {u}_k \in C^{1,\b}(\R)$, $\psi\in C^{0,\b}(\R)$ and $k\in \N $. Define the functions $\cF, \ti{\cF}: \R\to \R$ by 
\begin{align*}
\cF(s)&=\int_{S^{N-2}} \mu_{\a-2}(\s) \int_{\R}  t \L (\phi, s,t, p_\s)D_{u}^k {\cK}_\a({u},s ,t ,p_\s   )[u_1,\dots,u_k] \psi(s-p_\s t) dt  d\s
\end{align*}
and 
\begin{align*}
\ti \cF(s)&=\int_{S^{N-2}}\mu_{\a-3}(\s) \int_{\R}  D_{u}^k {\cK}_\a({u}, s ,t ,p_\s   )[u_1,\dots,u_k] \psi(s-p_\s t) dt  d\s.
\end{align*}
Then $\cF\in C^{0,\beta-\alpha}(\R)$ and $\ti \cF \in C^{0,\beta}(\R)$. Moreover, there exists a constant $c=c(N,\a,\b,k,\d)>1$ such that 
\begin{equation}
  \label{eq:est-F1}
\|\cF\|_{C^{0,\beta-\alpha}(\R)} \leq  c(1+ \|{u}\|_{C^{1,\b}(\R)} )^{c} \|\phi\|_{C^{1,\b}(\R)}   \|\psi\|_{C^{0,\b}(\R)}    \prod_{i=1}^k \|{u}_i\|_{C^{1, \b}(\R)} 
\end{equation}
and 
\begin{equation}
  \label{eq:est-F2}
\|\ti \cF\|_{C^{0,\beta}(\R)} \leq  c(1+ \|{u}\|_{C^{1,\b}(\R)} )^{c}    \|\psi\|_{C^{0,\b}(\R)}   \prod_{i=1}^k \|{u}_i\|_{C^{1, \b}(\R)}. 
\end{equation}
\end{lemma}
\begin{proof} 
Throughout this proof, the letter $c$ stands for different constants greater than one and depending only on $N,\alpha,\b,k$ and $\d$. We define
$$
F(s,t,p):=t  \L (\phi,s,t, p) D_{u}^k {\cK}_\a({u}, s ,t,p    ) [u_1,\dots,u_k] \psi(s-p t).
$$
We now use \eqref{eq:uv-s_1s_2},  the estimates \eqref{eq:Phis}, \eqref{eq:Phis1s2}, \eqref{eq:Dk-K-s}, 
\eqref{eq:Dk-K-s_1s_2} and the fact $\psi\in C^{0,\b}(\R) $. We also assume that $|s_1-s_2|\leq 1\leq 2$, which leads 
(since also $|p|\leq 2$) to $|p|^\b|s_1-s_2|^\b\le 2^\b \min (|p|^\b,|s_1-s_2|^\b)$.
We deduce that 
\begin{align}
&|[F(\cdot,t,p); s_1,s_2]|\le  c(1+ \|{u}\|_{C^{1,\b}(\R)} )^{c} \;\times
 \nonumber \\
&\Bigl(\frac{| t| \min( |t|^\b   |p|^\b, |s_1-s_2|^\b)}{    (1+ t^2)^{(N+\a)/2}    } 
+ \frac{|t|^{\b+1}   |p|^\b|s_1-s_2|^\b }{    (1+t^2)^{(N+\a)/2}    }\Bigr)
\|\phi\|_{C^{1,\b}(\R)}\,   \|\psi\|_{C^{0,\b}(\R)}   \prod_{i=1}^k \|{u}_i\|_{C^{1, \b}(\R)} 
 \nonumber \\ 
&\leq c(1+ \|{u}\|_{C^{1,\b}(\R)} )^{c}\frac{\min \{|p|^\beta,|s_1-s_2|^\beta\}}{(1+ t^2)^{\frac{N+\alpha-\beta-1}{2}}} \|\phi\|_{C^{1,\b}(\R)}\,   \|\psi\|_{C^{0,\b}(\R)}  \prod_{i=1}^k \|{u}_i\|_{C^{1, \b}(\R)},\nonumber
\end{align}
and consequently, since $N-1+\alpha -\beta>1$,  
\begin{align*}
|[ \cF; s_1,s_2]|&\le \int_{S^{N-2}} \mu_{\a-2}(\s) \int_{\R} |[F(\cdot,t,p_\s); s_1,s_2]|dtd\s\\
 &\le c(1+ \|{u}\|_{C^{1,\b}(\R)} )^{c}\|\phi\|_{C^{1,\b}(\R)}\,   \|\psi\|_{C^{0,\b}(\R)}    
 \prod_{i=1}^k \|{u}_i\|_{C^{1, \b}(\R)} \, \times\\
 &\qquad\qquad 
\int_{S^{N-2}} \mu_{\a-2}(\s) \min \{|p_\s|^\beta,|s_1-s_2|^\beta\}d\s.  
\end{align*}
We then have
\begin{align*}
\int_{S^{N-2}} &\mu_{\a-2}(\s) \min \{|p_\s|^\beta,|s_1-s_2|^\beta\}d\s \\
&= |S^{N-3}| \int_{-1}^1 \frac{(1-\sigma_1^2)^{\frac{N-4}{2}}}{(2-2\sigma_1)^{\frac{N+\alpha-2}{2}}}  \min \{(2-2\sigma_1)^{\beta/2},|s_1-s_2|^\beta\}\,d\sigma_1\\
&\le c \int_{-1}^1 (2-2\sigma_1)^{-\alpha/2-1}  \min \{(2-2\sigma_1)^{\beta/2} ,|s_1-s_2|^\beta\}
\,d\sigma_1\\
&\le c \int_{0}^4 \tau^{-\alpha/2-1}  \min \{\tau^{\beta/2},|s_1-s_2|^\beta\}d \tau\\
&\le c \Bigl( \int_0^{|s_1-s_2|^2} \tau^{\frac{\beta-\alpha}{2}-1}\,d\tau + |s_1-s_2|^\beta \int_{|s_1-s_2|^2}^4 \tau^{-\alpha/2-1}\,d\tau\Bigr) \le c |s_1-s_2|^{\beta-\alpha }, 
\end{align*}
so that 
$$
|[ \cF; s_1,s_2]| \le c(1+ \|{u}\|_{C^{1,\b}(\R)} )^{c}  \|\phi\|_{C^{1,\b}(\R)}\,   \|\psi\|_{C^{0,\b}(\R)}|s_1-s_2|^{\beta-\alpha } \prod_{i=1}^k \|{u}_i\|_{C^{1, \b}(\R)}. 
$$
Similarly but more easily, we also obtain the estimate   
$$
\|\cF\|_{L^\infty(\R)} \le c(1+ \|{u}\|_{C^{1,\b}(\R)} )^{c} \|\phi\|_{C^{1,\b}(\R)}\,   \|\psi\|_{C^{0,\b}(\R)}  \prod_{i=1}^k \|{u}_i\|_{C^{1, \b}(\R)}, 
$$
and thus (\ref{eq:est-F1}) follows. 

To prove (\ref{eq:est-F2}), we now set 
$$
\ti{F}(s,t,p):= D_{u}^k {\cK}_\a({u},s - p t,s ,t    )[u_1,\dots,u_k]  \psi(s-pt),
$$
and we get
\begin{align*}
|[\ti{F}(\cdot,t,p); s_1,s_2]| \leq  
 c(1+ \|{u}\|_{C^{1,\b}(\R)} )^{c} \frac{ |s_1-s_2|^\b }{(1+ t^2)^{(N+\a)/2} }  \|\psi\|_{C^{0,\b}(\R)}    \prod_{i=1}^k \|{u}_i\|_{C^{1, \b}(\R)}  
\end{align*}
and 
\begin{align*}
\|[\ti{F}(\cdot,t,p)\|_{L^\infty(\R)} \leq  
 c \frac{(1+ \|{u}\|_{C^{1,\b}(\R)} )^{c}}{(1+ t^2)^{(N+\a)/2} }  \|\psi\|_{C^{0,\b}(\R)}     
 \prod_{i=1}^k \|{u}_i\|_{C^{1, \b}(\R)}.
\end{align*}
By \eqref{eq:int-mu_n},  we thus have (\ref{eq:est-F2}).
 \end{proof}
 
 With the aid of this lemma we can now prove the following result.
 
\begin{lemma}\label{prop:smooth-ovH-0}
Let $N\geq 3$, $k \in \N \cup \{0\}$, $\d>0$, $u \in \cO_\d$, and $u_1,\dots,u_k \in C^{1,\beta}(\R)$. Moreover, let $\cM,\ti \cM: \R \to \R$ be defined by 
\begin{align*}
\cM(s)&= \int_{S^{N-2}} \mu_{\a-2}(\s) \int_\R D^k_u {M}({u} , s,t,p_\s)[u_1,\dots,u_k]\,dt d\s,\\
\ti \cM(s)&= \int_{S^{N-2}} \mu_{\a-3}(\s) \int_\R D^k_u {\ti M}({u} , s,t,p_\s)[u_1,\dots,u_k]\,dt d\s,
\end{align*}
where $M, \, \ti{M}:   \cO_\d \times \R^3    \to \R$  are given by 
\begin{align*}
M({u} , s,t,p)&= t \L ({u}, s,t, p)\cK_\a({u}, s,t,p)   \, {u}^{N-2}(s-p t)\qquad \text{and}\\
\ti{M}({u} , s,t,p)&=  u(s)\cK_\a({u}, s,t,p)   \, {u}^{N-2}(s-p t).    
\end{align*}
Then $\cM \in C^{0,\beta-\alpha}(\R)$ and  $\ti \cM \in C^{0,\beta}(\R)$. Moreover, there exists a constant $c=c(N,\a,\b,k,\d)>1$ such that 
$$
\|\cM\|_{C^{0,\beta-\alpha}(\R)} \leq  c(1+ \|{u}\|_{C^{1,\b}(\R)} )^{c}   \prod_{i=1}^k \|{u}_i\|_{C^{1, \b}(\R)} 
$$
and 
\begin{equation}
  \label{eq:est-M2}
\|\ti \cM\|_{C^{0,\beta}(\R)} \leq  c(1+ \|{u}\|_{C^{1,\b}(\R)} )^{c}   \prod_{i=1}^k \|{u}_i\|_{C^{1, \b}(\R)}.
\end{equation}
\end{lemma}
\begin{proof}
We define $T: \cO_\d \times \R^3    \to \R$ by 
$$ 
T({u} , s,t,p)= t \L ({u}, s,t, p)\cK_\a({u}, s,t,p),  
$$
so that 
$$
M({u} , s,t,p)= T({u} , s,t,p) \, {u}^{N-2}(s-p t)  .
$$
 By \eqref{eq:Dk-T1T2}, we thus have 
 $$
 D^k_{u} M({u} , s,t,p)[u_1,\dots,u_k]  = \sum_{\cN \in  \scrS_k}\psi_{\cN}(s-p t)  D^{|\cN|}_u T({u} , s,t,p)[u_i]_{i \in \cN}   \,  
 $$
 where $\psi_{\cN}:=u^{N-2}$ in case $k=|\cN|$ and 
$$
\Psi_{\cN}:= \prod_{\ell=0}^{k-|\cN|-1} (N-2-\ell)\: {u}^{N-2-(k-|\cN |)} \prod_{i \in \cN^c}  {u}_{i} \qquad \text{in case $k>|\cN|$} 
$$
(noting that $|\cN^c|=k -|\cN|$). By \eqref{eq:Dk-LT2} we have, if $|\cN| \ge 1$, 
 \begin{align*}
 D^{|\cN|}_{u} T({u} , s,t,p)[u_i]_{i \in \cN} &=  t \L({u},s,t,p)    D^{|\cN| }_{u}\cK_\a({u},s,t,p)[u_i]_{i \in \cN} \\
&\hspace{5mm}+  \sum_{j\in\cN} t  \L({u}_{j},s,t,p)   D^{|\cN|-1}_{u} \cK_\a({u},s,t,p) 
 [u_i]_{\stackrel{i \in \cN}{ i\neq j}}.
 \end{align*}
Consequently, 
$$
 D^k_{u} M({u} , s,t,p)[u_1,\dots,u_k] = \sum_{\cN \in     \scrS_k} M_{\cN}(s,t,p)
$$
with 
\begin{align*}
M_{\cN}(s,t,p)=& t \psi_{\cN}(s-p t) \Bigl(  \L({u},s,t,p)    D^{|\cN| }_{u}\cK_\a({u},s,t,p) [u_i]_{i \in \cN}\\
&+  \sum_{j\in\cN}  \L({u}_{j},s,t,p)   D^{|\cN|-1}_{u} \cK_\a({u},s,t,p) [u_i]_{\stackrel{i \in \cN}{ i\neq j}}\Bigr).
\end{align*}
Clearly we also have that 
$$
\| \psi_{\cN} \| _{C^{0,\b}(\R)} \leq c (1+ \|{u}\|_{C^{1,\b}(\R)})^{c}    \prod_{i \in \cN^c} \| {u}_{i}\| _{C^{1,\b}(\R)}.
$$
By Lemma~\ref{lem:est-cand-deriv}, it thus follows that $\cM \in C^{0,\beta-\alpha}(\R)$ and 
\begin{align*}
\|\cM\|_{C^{0,\beta-\alpha}(\R)} &\leq  c(1+ \|{u}\|_{C^{1,\b}(\R)} )^{c} \sum_{\cN   \in     \scrS_k} \|\psi_\cN\|_{C^{0,\b}(\R)} 
   \prod_{i \in \cN} \|{u}_i\|_{C^{1, \b}(\R)}   \\
 &\leq  c(1+ \|{u}\|_{C^{1,\b}(\R)} )^{c}   \prod_{i=1}^k\|{u}_i\|_{C^{1, \b}(\R)},
\end{align*}
as claimed.  The proof of (\ref{eq:est-M2}) is similar but easier. 
\end{proof}

We are now in position to prove that $\cH: \cO_\d\to C^{0,\b-\a}(\R)$ given by \eqref{eq:good-expression2-NMC} is smooth. 
\begin{proposition}
\label{prop:smooth-ovH}  
For $N\geq 3$, the map  $\cH: \cO_{\d} \subset C^{1,\beta}(\R) \to C^{0, \b-\a}(\R)$ defined by (\ref{eq:good-expression2-NMC}) is of class $C^\infty$, and for every  $k\in\N$ we have      
\begin{align}
D^k \cH({u} )=\int_{S^{N-2}}&\!\mu_{\a-2}(\s) \!\!\int_{\R}  D^k_u M(u,\cdot ,t,p_\s)dt  d\s \nonumber\\
&- \frac{1}{2}\int_{S^{N-2}}\!\mu_{\a-3}(\s)\!\! \int_{\R}  D^k_u \ti{M}(u,\cdot ,t,p_\s)dt  d\s,\nonumber\label{eq:Dk-ovH}
\end{align}
where $M$ and $\ti M$ are defined in Lemma~\ref{prop:smooth-ovH-0}. 
\end{proposition}

\begin{proof}
We can write ${\cH} = {\cH}_1 - {\cH}_2$ with 
$$
{\cH}_1(u)(s) = \int_{S^{N-2}}\mu_{\a-2}(\s) \int_{\R}    M(u,s,t,p_\s)dt  d\s
$$
and 
$$ 
{\cH}_2(u)(s) = \frac{1}{2}\int_{S^{N-2}}\mu_{\a-3}(\s) \int_{\R}    \ti{M}(u,s,t,p_\s)dt  d\s.
$$
We only prove that, for $k \in \N \cup \{0\}$,  
\begin{equation}
  \label{eq:statement-H-1}
D^k {\cH}_1({u} )= \int_{S^{N-2}}\mu_{\a-2}(\s) \int_{\R}  D^k_u M(u,\cdot ,t,p_\s)dt  d\s \qquad \text{in the Fr\'echet sense.}   
\end{equation}
The corresponding statement for ${\cH}_2$ is similar but simpler to prove. Moreover, the continuity of $D^k {\cH}$ is a well known consequence of the existence of $D^{k+1}{\cH}$ in the Fr{\'e}chet sense.

 To prove (\ref{eq:statement-H-1}), we proceed by induction. For $k = 0$, the statement is true by definition. 
 Let us now assume that the statement holds true for some $k \ge 0$. Then $D^k {\cH}_1({u} )$ is given by 
 $$
 D^{k} {\cH}_1({u})[u_1,\dots,u_{k}] (s)=\int_{S^{N-2}}\mu_{\a-2}(\s)\int_\R  D^{k}_{{u}} M({u},s,t,p_\s)[u_1,\dots,u_{k}]\,  dt d\s.
$$
We fix ${u}_1,\dots,{u}_{k} \in C^{1,\b}(\R)$. For  ${u}\in \cO_\d $ and $v \in C^{1,\beta}(\R)$,  we define
$$
\G({u},v,s):=\int_{S^{N-2}}\mu_{\a-2}(\s)\int_\R  D^{k+1}_{u} M({u},s,t,p_\s)[u_1,\dots,u_k,v]\,  dt d\s.
$$
Let $u\in\cO_\d$ and $v\in  C^{1,\beta}(\R)  $  with  $\|v\|_{C^{1,\beta}(\R)}<\d/2$. We have  
\begin{align*}
D^{k}&   {\cH}_1({u}+v)[u_1,\dots,u_{k}](s)-  D^{k}  {\cH}_1({u} )[u_1,\dots,u_k] (s)- \G({u},v , s)
 \nonumber \\
 =&\int_{S^{N-2}}\mu_{\a-2}(\s) \;\times \nonumber\\
&\int_\R \int_0^1 \!\! \left\{ D^{k+1}_{u}  M({u}+\rho v ,s,t,p_\s)-D^{k+1}_{u} M({u},s,t,p_\s)\right\} \!\![u_1,\dots,u_k,v] d\rho dt d\s\nonumber \\
 =& \int_0^1 \rho \int_0^1  {\cH}_1^{\rho,\tau}(s)d\t d\rho,
\end{align*} 
with 
$$
{\cH}_1^{\rho,\tau}(s):= \int_{S^{N-2}} \mu_{\a-2}(\s) \int_{\R}  D^{k+2}_{u} M({u}+\t \rho  v,s,t,p_\s)[u_1,\dots,u_k,v,v]  dt d\s.
$$
Note that ${u}+\t \rho  v\in  \cO_{\d/2}$ for every $\t,\rho\in[0,1]$.    By Lemma~\ref{prop:smooth-ovH-0}, we have 
\begin{align*}
\|{\cH}_1^{\rho,\tau}\|_{C^{0,\b-\a}(\R)}  &\leq  c(1+ \|u+\t \rho  v \|_{C^{1,\b}(\R)} )^c \|v\|_{C^{1, \b}(\R)}^2 \prod_{i=1}^k \|u_i\|_{C^{1,\beta}(\R)}\\
&\leq  c(1+ \|u\|_{C^{1,\b}(\R)} +\| v \|_{C^{1,\b}(\R)})^c  \|v\|_{C^{1, \b}(\R)}^2 \prod_{i=1}^k \|u_i\|_{C^{1,\beta}(\R)},
\end{align*}
with a constant $c>1$ independent of $\rho,\tau,u,u_1,\dots,u_{k}$ and $v$. Consequently, 
\begin{align*}
\| D^{k}   &{\cH}_1({u}+v)[u_1,\dots,u_{k}]-  D^{k}  {\cH}_1({u} )[u_1,\dots,u_k]- \G({u},v , \cdot)\|_{C^{0,\b-\a}(\R)}\\
&\leq  c(1+ \|u\|_{C^{1,\b}(\R)} +\| v \|_{C^{1,\b}(\R)})^c  \|v\|_{C^{1, \b}(\R)}^2 \prod_{i=1}^k \|u_i\|_{C^{1,\beta}(\R)}.
\end{align*}
This shows that $D^{k+1}{\cH}_1(u)$ exists in the Frech\'et sense, and that 
$$
 D^{k+1} {\cH}_1(u)[u_1,\dots,u_{k},v] = \G(u,v,\cdot) \quad \in \; C^{0,\beta-\alpha}(\R).
$$
We conclude that (\ref{eq:statement-H-1}) holds for $k+1$ in place of $k$, and thus the proof is finished.
\end{proof}

We finally establish the promised expression for the differential
of $H$ at constant functions. By this we complete the proof of Proposition~\ref{lem:diff-H}.
\begin{proposition}
\label{cor-1st-der}
Let $N\geq 3$. If $u \equiv \kappa \in \cO_\d$ is a constant function, we have 
$$
 -\frac{\alpha}{2}D H(\kappa)v(s)= D{\cH}(\kappa)v(s)= - \frac{\alpha}{2} \kappa^{-1-\alpha} \Bigl(  PV \int_{\R} (v(s)-v(s-\kappa \tau)) G_\alpha(\tau)\,d\tau - b_\alpha v(s)\Bigr)  
$$
with 
$$
G_\alpha: \R \setminus \{0\} \to \R, \qquad G_\alpha(\tau)= \int_{S^{N-2}}\frac{2}{(\tau^2+2(1-\s_1))^{\frac{N+\a}{2}}}\,d\s
$$
and 
$$
b_\alpha = 2 \int_{\R} \int_{S^{N-2}} \frac{1-\sigma_1}{(\tau^2  + 2(1-\sigma_1))^{\frac{N+\alpha}{2}}}\,d\s d\tau.
$$
\end{proposition}

\begin{proof}
Proposition~\ref{prop:smooth-ovH} gives the formula 
\begin{align}
  \label{eq:formula-first-der}
D{\cH}(u)v(s) = \int_{S^{N-2}}\!\mu_{\a-2}(\s) \!\!\int_{\R}  D_u M(u,s,t,p_\s)v\, dt  d\s\nonumber\\
- \frac{1}{2} \int_{S^{N-2}}\!\mu_{\a-3}(\s) \!\!\int_{\R}  D_u \ti M(u,s,t,p_\s)v\,dt  d\s.  
\end{align}
In the case where $u \equiv \kappa \in \cO_\d$ is a constant function, we have $\L (\kappa, s,t, p) = 0$
and thus 
\begin{align*}
D_u M(\kappa,s ,t,p)v &=   t  \L ({v}, s,t, p)\cK_\a(\kappa, s,t,p)\kappa^{N-2}\\
&= t \Bigl(\frac{v(s)-v(s-pt)}{pt}-v'(s-pt)\Bigr) \cK_\a(\kappa, s,t,p) \kappa^{N-2}\\
&= t \Bigl(\frac{v(s)-v(s-pt)}{pt}-v'(s-pt)\Bigr) \frac{\kappa^{N-2}}{(t^2 + \kappa^2)^{\frac{N+\alpha}{2}}}.
\end{align*}
Therefore, by substituting $\tau = \frac{p_\s}{\kappa}t$,  
\begin{align*}
\int_{S^{N-2}}&\!\mu_{\a-2}(\s) \!\!\int_{\R}  D_u M(u,s ,t,p_\s)v \,dt  d\s =\\ 
&=\kappa^{N-2} 
\int_{S^{N-2}} p_\s^{2-N-\a}  \!\!\int_{\R}  t \Bigl(\frac{v(s)-v(s-p_\s t)}{p_\s t}-v'(s-p_\s t)\Bigr) \frac{1}{(t^2  + \kappa^2)^{\frac{N+\alpha}{2}}}  dt  d\s\\  
&=\kappa^{-\alpha}  \int_{S^{N-2}}   \!\!\int_{\R} \tau \Bigl(\frac{v(s)-v(s-\kappa \tau)}{\kappa \tau}-v'(s-\kappa \tau)\Bigr) \frac{1}{(\tau^2  + p_\s^2)^{\frac{N+\alpha}{2}}}  d\tau  d\s\\  
&=\kappa^{-1-\alpha}  \int_{S^{N-2}} \int_{\R}   \frac{v(s)-v(s-\kappa \tau)- \kappa \tau v'(s-\kappa \tau)}{(\tau^2  + p_\s^2)^{\frac{N+\alpha}{2}}}  d\tau  d\s\\
&=\kappa^{-1-\alpha} \lim_{\eps \to 0}  \int_{S^{N-2}} \int_{|\tau| \ge \eps}   \frac{v(s)-v(s-\kappa \tau)- \kappa \tau v'(s-\kappa \tau)}{(\tau^2  + p_\s^2)^{\frac{N+\alpha}{2}}}  d\tau  d\s,
\end{align*}
whereas, by integration by parts,  
\begin{align}
&\int_{|\tau| \ge \eps}  \frac{\kappa \tau v'(s-\kappa \tau)}{(\tau^2  + p_\s^2)^{\frac{N+\alpha}{2}}} d\tau = -\int_{|\tau| \ge \eps} \partial_\tau v(s-\kappa \tau)  \frac{\tau }{(\tau^2  + p_\s^2)^{\frac{N+\alpha}{2}}} d\tau \label{int-by-parts-tau-formula}\\
&= \frac{\eps (v(s-\kappa \eps)+v(s+\kappa \eps))}{(\eps^2  + p_\s^2)^{\frac{N+\alpha}{2}}}  + \int_{|\tau| \ge \eps} v(s-\kappa \tau) \partial_\tau \Bigl(\frac{\tau }{(\tau^2  + p_\s^2)^{\frac{N+\alpha}{2}}}\Bigr) d\tau \nonumber\\
&= \frac{\eps (v(s-\kappa \eps)+v(s+\kappa \eps))}{(\eps^2  + p_\s^2)^{\frac{N+\alpha}{2}}} + \int_{|\tau| \ge \eps} v(s-\kappa \tau)\Bigl(\frac{1}{(\tau^2  + p_\s^2)^{\frac{N+\alpha}{2}}}-\frac{(N+\alpha)\tau^2}{(\tau^2  + p_\s^2)^{\frac{N+\alpha+2}{2}}}\Bigr) d\tau \nonumber \\
&=\frac{\eps (v(s-\kappa \eps)+v(s+\kappa \eps))}{(\eps^2  + p_\s^2)^{\frac{N+\alpha}{2}}} + \int_{|\tau| \ge \eps} v(s-\kappa \tau)\Bigl( \frac{(N+\alpha) p_\s^2}{(\tau^2  + p_\s^2)^{\frac{N+\alpha+2}{2}}} -\frac{N+\alpha-1}{(\tau^2  + p_\s^2)^{\frac{N+\alpha}{2}}}\Bigr) d\tau. \nonumber
\end{align}
Hence, by Fubini's theorem,  
\begin{align*}
&\int_{S^{N-2}}\!\mu_{\a-2}(\s) \!\!\int_{\R}  D_u M(\kappa,s,t,p_\s)v dt  d\s\\
&=  \kappa^{-1-\alpha} \lim_{\eps \to 0} \Biggl( v(s)  \int_{|\tau| \ge \eps} \widetilde G_\alpha(\tau)d\tau - \eps \{v(s-\kappa \eps)+v(s+\kappa \eps)\} \widetilde G_\alpha(\eps)  
\\
&\hspace{6mm}+ v(s-\kappa \tau) \Bigl( (N+\alpha-2)  \int_{|\tau| \ge \eps} \widetilde G_\alpha(\tau) \,d\tau -  (N+\alpha) \int_{|\tau| \ge \eps} G_{1,\alpha}(\tau) \,d\tau \Bigr)\Biggr).
\end{align*}
Here and in the following, we put 
\begin{align*}
\widetilde G_\alpha(\tau)&= \frac{G_\alpha(\tau)}{2}=  \int_{S^{N-2}} \frac{1}{(\tau^2  + p_\s^2)^{\frac{N+\alpha}{2}}}d\s,\quad \qquad G_{\alpha,0}(\tau) :=  \int_{S^{N-2}} \frac{p_\s^2}{(\tau^2  + p_\s^2)^{\frac{N+\alpha}{2}}}d\s\\
G_{1,\alpha}(\tau)&= \int_{S^{N-2}} 
\frac{p_\s^2}{(\tau^2  + p_\s^2)^{\frac{N+\alpha+2}{2}}}\,d\s,\quad \qquad \qquad \quad G_{\alpha,2}(\tau):= \int_{S^{N-2}} \frac{p_\s^4}{(\tau^2  + p_\s^2)^{\frac{N+\alpha+2}{2}}}\,d\s.
\end{align*}

We also have 
\begin{align*}
D_u& \ti M(\kappa ,s ,t,p)v =   v(s) \cK_\a(\kappa, s,t,p) \kappa^{N-2} + \kappa \cK_\a(\kappa, s,t,p) (N-2) \kappa^{N-3}v(s-pt)  \\
&\hspace{40mm} + \kappa^{N-1}  D_u  \cK_\a(\kappa, s,t,p)v\\
&=   \kappa^{N-2} \Bigl(  (v(s) +(N-2)v(s-pt)) \cK_\a(\kappa, s,t,p) + \kappa D_u  \cK_\a(\kappa, s,t,p)v\Bigr)\\
&=   \kappa^{N-2} \Bigl( \frac{v(s) +(N-2)v(s-pt)}{(t^2+\kappa^2)^{\frac{N+\alpha}{2}}} - \frac{N+\alpha}{2} \, \frac{\kappa^2(v(s)+v(s-pt))}{(t^2 + \kappa^2)^{\frac{N+\alpha+2}{2}}} \Bigr)\\
&=   \kappa^{-2-\alpha}p^{N+\alpha}  \Bigl( \frac{v(s) +(N-2)v(s-pt)}{(\frac{p^2}{\kappa^2} t^2+p^2)^{\frac{N+\alpha}{2}}} - \frac{N+\alpha}{2} \, \frac{p^2(v(s)+v(s-pt))}{(\frac{p^2}{\kappa^2} t^2+p^2)^{\frac{N+\alpha+2}{2}}} \Bigr), 
\end{align*}
so that we have, again by substituting $\tau = \frac{p_\s}{\kappa}t$, 
\begin{align*}
&\int_{S^{N-2}}\!\mu_{\a-3}(\s) \!\!\int_{\R}  D_u \ti M(\kappa ,s ,t,p_\s)v\,dt  d\s = \kappa^{-1-\alpha} \:\times \\
&\hspace{10mm}\int_{S^{N-2}} \int_{\R}     \Bigl(\frac{p_\s^{2}(v(s) +(N-2)v(s-\kappa t))}{(\tau^2+p_\s^2)^{\frac{N+\alpha}{2}}} - \frac{N+\alpha}{2} \,\frac{p_\s^4(v(s)+v(s-\kappa \tau))}{(\tau^2+p_\s^2)^{\frac{N+\alpha+2}{2}}} \Bigr) d \tau d \s\\
&= \kappa^{-1-\alpha} \lim_{\eps \to 0}\Biggl( \int_{|\tau| \ge \eps}  \bigl\{v(s) +(N-2)v(s-\kappa \t)\bigr\} G_{\alpha,0}(\tau)\,d\tau\\
&\qquad\qquad \hspace{10mm}- \frac{N+\alpha}{2} 
\int_{|\tau| \ge \eps} \bigl\{v(s)+v(s-\kappa \tau)\bigr\}G_{\alpha,2}(\tau)  \,d\tau\Biggr),
\end{align*}
where in the last step we used Lebesgue's and Fubini's theorems. 
Collecting and reordering everything, and recalling (\ref{eq:formula-first-der}), we thus get 
\begin{align*}
&\kappa^{1+\alpha} D{\cH}(\kappa)v(s) \\
&= \lim_{\eps \to 0} \Biggl[  \alpha \int_{|\tau| \ge \eps}(v(s-\kappa \tau)-v(s))\widetilde G_\alpha(\tau)\,d\tau 
+\alpha \frac{v(s)}{2} \int_{|\tau| \ge \eps} G_{\alpha,0}(\tau)\,d\tau \\
&\hspace{15mm} + v(s) \int_{|\tau| \ge \eps}  \Bigl(\frac{N+\alpha}{4} G_{\alpha,2}(\tau)-\frac{\alpha+1}{2}G_{\alpha,0}(\tau)  + (\alpha+1) \widetilde G_\alpha(\tau) \Bigr)\,d\tau\\
&\hspace{15mm}-\eps \{v(s-\kappa \eps)+v(s+\kappa \eps)\} \widetilde G_\alpha(\eps) \\  
&\hspace{15mm}+ \int_{|\tau| \ge \eps} v(s-\kappa \tau)  \Bigl( (N-2) \widetilde G_\alpha(\tau) -
\frac{N-2}{2}G_{\alpha,0}(\tau) \\
&\hspace{60mm} -(N+\alpha)\bigl \{G_{1,\alpha}(\tau)-\frac{G_{\alpha,2}(\tau)}{4}\bigr\}\Bigr) d\tau\Biggr]. 
\end{align*}

We now claim that for every $\tau \in \R \setminus \{0\}$ we have 
\begin{equation}
  \label{eq:S-N-2-identity}
(N-2) \widetilde G_\alpha(\tau) -\frac{N-2}{2}G_{\alpha,0}(\tau)  -  (N+\alpha)\bigl \{G_{1,\alpha}(\tau)-\frac{G_{\alpha,2}(\tau)}{4}\bigr\}= 0.
\end{equation}
Indeed,  
\begin{align*}
&(N+\alpha)\bigl \{G_{1,\alpha}(\tau)-\frac{G_{\alpha,2}(\tau)}{4}\bigr\} = (N+\alpha)\int_{S^{N-2}} \frac{ p_\s^2 - \frac{p_\s^4}{4} }{(\tau^2  + p_\s^2)^{\frac{N+\alpha+2}{2}}}d\s\\
&=|S^{N-3}| (N+\alpha) \int_{-1}^1 \frac{(1-\sigma_1^2)^{\frac{N-4}{2}}(2-2\sigma_1)(\frac{1}{2}+\frac{\sigma_1}{2}) }{(\tau^2 + (2-2\sigma_1))^{\frac{N+\alpha+2}{2}}}d \sigma_1\\
&=|S^{N-3}| (N+\alpha) \int_{-1}^1 \frac{(1-\sigma_1^2)^{\frac{N-2}{2}}}{(\tau^2 + (2-2\sigma_1))^{\frac{N+\alpha+2}{2}}}d \sigma_1\\
&= |S^{N-3}| \int_{-1}^1(1-\sigma_1^2)^{\frac{N-2}{2}} \partial_{\sigma_1}\frac{1}{(\tau^2 + (2-2\sigma_1))^{\frac{N+\alpha}{2}}}d \sigma_1\\
&=-|S^{N-3}| \int_{-1}^1 \frac{\partial_{\sigma_1} (1-\sigma_1^2)^{\frac{N-2}{2}}}{(\tau^2 + (2-2\sigma_1))^{\frac{N+\alpha}{2}}}d \sigma_1= (N-2) |S^{N-3}| \int_{-1}^1 \frac{\sigma_1 (1-\sigma_1^2)^{\frac{N-4}{2}}}{(\tau^2 + (2-2\sigma_1))^{\frac{N+\alpha}{2}}}d \sigma_1\\
&=- \frac{N-2}{2} \int_{S^{N-2}} \frac{p_\s^2-2}{(\tau^2+ p_\s^2)^{\frac{N+\alpha}{2}}}d\s=- \frac{N-2}{2} G_{\alpha,0}(\tau) + (N-2)\widetilde G_\alpha(\tau). 
\end{align*}
Hence we conclude that  
\begin{align*}
&\kappa^{1+\alpha} D{\cH}(\kappa)v(s) + \alpha \Bigl( PV \int_{\R}(v(s)-v(s-\kappa \tau))\widetilde G_\alpha(\tau)\,d\tau -\frac{v(s)}{2} \int_{|\tau| \ge \eps} G_{\alpha,0}(\tau)\,d\tau \Bigr)\\
&= \lim_{\eps \to 0}R_\eps(s),
\end{align*}
where
\begin{align*}
R_\eps(s) :=& v(s) \int_{|\tau| \ge \eps}  \Bigl(\frac{N+\alpha}{4} G_{\alpha,2}(\tau)-\frac{\alpha+1}{2}G_{\alpha,0}(\tau)  + (\alpha+1) \widetilde G_\alpha(\tau) \Bigr)\,d\tau\\
&-\eps \{v(s-\kappa \eps)+v(s+\kappa \eps)\} \widetilde G_\alpha(\eps).   
\end{align*}

The proof of the proposition is finished once we have shown that $\lim \limits_{\eps \to 0}R_\eps(s)= 0$. To see this, we note that by choosing $v \equiv 1$ in (\ref{int-by-parts-tau-formula}) we have the identity
\begin{equation}
  \label{eq:add-identity}
\int_{|\tau| \ge \eps} \frac{N+\alpha-1}{(\tau^2  + p_\s^2)^{\frac{N+\alpha}{2}}}d\tau  =  
\int_{|\tau| \ge \eps} \frac{(N+\alpha)p_\s^2}{(\tau^2  + p_\s^2)^{\frac{N+\alpha+2}{2}}}d\tau +  \frac{2\eps}{(\eps^2  + p_\s^2)^{\frac{N+\alpha}{2}}}
\end{equation}
for $\s \in S^{N-2}$. Integrating this identity over $S^{N-2}$ yields 
$$
(N+\alpha-1) \int_{|\tau| \ge \eps} \widetilde G_\alpha(\tau) d\tau =  (N+\alpha) \int_{|\tau| \ge \eps} G_{1,\alpha} (\tau) d\tau  + 2 \eps \,\widetilde G_\alpha(\eps). 
$$
On the other hand, multiplying (\ref{eq:add-identity}) with $p_\s^2$ and integrating over $S^{N-2}$ yields
$$
(N+\alpha-1) \int_{|\tau| \ge \eps} G_{\alpha,0}(\tau) d\tau =  (N+\alpha) \int_{|\tau| \ge \eps} G_{\alpha,2} (\tau) d\tau  + 2 \eps \, G_{\alpha,0}(\eps).   
$$
Inserting the two previous identities successively gives 
\begin{align*}
R_\eps(s) &= v(s) \int_{|\tau| \ge \eps}  \Bigl(\frac{N+\alpha}{4} G_{\alpha,2}(\tau)-\frac{\alpha+1}{2}G_{\alpha,0}(\tau)  + (\alpha+1) \widetilde G_\alpha(\tau) \Bigr)\,d\tau\\
&\hspace{10mm}-\eps \{v(s-\kappa \eps)+v(s+\kappa \eps)\} \widetilde G_\alpha(\eps)\\
&= v(s) \int_{|\tau| \ge \eps}  \Bigl((N+\alpha) \frac{G_{\alpha,2}(\tau)}{4} +(N+\alpha)  G_{1,\alpha} (\tau) \\
&\hspace{40mm}
 - \frac{\alpha+1}{2}G_{\alpha,0}(\tau)  -  (N-2) \widetilde G_\alpha(\tau) \Bigr)\,d\tau\\
&\hspace{10mm}+ \eps \{2 v(s)- v(s-\kappa \eps)-v(s+\kappa \eps)\} \widetilde G_\alpha(\eps)\\
&= v(s) \int_{|\tau| \ge \eps}  \Bigl((N+\alpha) (G_{1,\alpha} (\tau) - \frac{G_{\alpha,2}(\tau)}{4}) 
 + \frac{N-2}{2}G_{\alpha,0}(\tau)  -  (N-2) \widetilde G_\alpha(\tau) \Bigr)\,d\tau\\
&\hspace{10mm} + \eps \{2 v(s)- v(s-\kappa \eps)-v(s+\kappa \eps)\} \widetilde G_\alpha(\eps)- \eps G_{\alpha,0}(\eps)\\ 
&= \eps \{2 v(s)- v(s-\kappa \eps)-v(s+\kappa \eps) \} \widetilde G_\alpha(\eps)- \eps G_{\alpha,0}(\eps), 
\end{align*}
where we have used (\ref{eq:S-N-2-identity}) again in the last step. Since 
$$
\widetilde G_\alpha(\eps)=  O(\eps^{-2-\alpha}), \qquad G_{\alpha,0}(\eps) = O(\eps^{-\alpha})
$$
and 
$$
2 v(s)- v(s-\kappa \eps)-v(s+\kappa \eps) = O(\eps^{1+\beta})
$$
as $\eps \to 0$ since $v \in C^{1,\beta}(\R)$, we conclude that 
$$
\lim_{\eps \to 0} R_\eps(s)= 0,
$$
as desired. 
\end{proof}

\section{Smooth branch of periodic  bands with constant nonlocal mean curvature}
The aim of this section is to derive the regularity of the nonlocal mean curvature operator $H$ when $N=2$,
thereby deducing  
the smoothness of the branch of CNMC  bands bifurcating from the straight one. We proved that this branch is continuous in \cite{Cabre2015A}, but there we did not prove its smoothness. For this, we follow the approach of Section~4.

In case $N=2$, from Lemma~\ref{lem:geomNMC} (with $S^{N-2}=S^0=\{-1,1\}\subset\R$),  
we deduce that  the nonlocal mean curvature $H_{E_u}$ at 
the point  $(s,u(s))  $ is given by
\begin{align*}
  -\frac{\alpha}{2}  H({u})(s) =&    \int_{\R}
  \frac{
  u(s)-u(s-\tau)-\tau u'(s-\tau)   
}{
\{\t^2+({u}(s)-{u}(s-\t))^2  \}^{(2+\a)/2}
}
d\tau \nonumber \\
& +   \int_{\R}
  \frac{
  u(s)-u(s-\tau)-\tau u'(s-\tau) 
}{
\{\t^2+({u}(s)-{u}(s-\t))^2+ 4{u}(s){u}(s-\t)  \}^{(2+\a)/2}
}
d\tau     \label{new-rep-H-2D} \\
& -
 2 {{u}(s)}    \int_{\R}
 \frac{1
}{
\{\t^2+({u}(s)-{u}(s-\t))^2+4 {u}(s){u}(s-\t)  \}^{(2+\a)/2}
}
d\t .  \nonumber   
\end{align*} 
This is a quite different expression than the one we used in \cite{Cabre2015A}.

In Lemma \ref{lem:est-cand-deriv-2D} we will see that the integrals above converge absolutely in the Lebesgue sense.
Changing $\t$ to $t$ and using the notation from the beginning of Section  \ref{sec:preparations-1}, we have 
\begin{align*}
  -\frac{\alpha}{2}  H({u})(s) =&    \int_{\R}
  \frac{ t\L ({u}, s,t, 1)  
}{|t|^{2+\a}
\left(1+\L_0(u,s,t,1)^2  \right)^{(2+\a)/2}
}
d t \nonumber \\
& +   \int_{\R}
  \frac{
 t\L ({u}, s,t, 1)
}{
\left(t^2+t^2 \L_0(u,s,t,1)^2+ 4{u}(s){u}(s-t)  \right)^{(2+\a)/2}
}
dt   \\
& -
 2 {{u}(s)}    \int_{\R}
 \frac{1
}{
\left(t^2+t^2 \L_0(u,s,t,1)^2+4 {u}(s){u}(s-t)  \right)^{(2+\a)/2}
}
dt.  \nonumber   
\end{align*} 
For $\a>0$, we define 
 the maps  $\cK_{\a, 0}: C^{1,\b}(\R)\times \R\times  \R \to \R $ and  $\cK_{\a, 1}:\cO\times \R\times   \R \to \R $  by 
$$ 
\cK_{\a, 0}(u,s,t):=\frac{1}{ \left(1+  \L_0(u,s,t,1)^2  \right)^{(2+\a)/2} }
$$
and 
$$ 
\cK_{\a, 1}(u,s,t):=\frac{1}{ \left(t^2+t^2 \L_0(u,s,t,1)^2+4   {u}(s){u}(s-t)   \right)^{(2+\a)/2}  }.
$$
Therefore, for every $u\in \cO$, we have 
\begin{align*}
  -\frac{\alpha}{2}  H({u})(s) =&    \int_{\R}
  \frac{ t\L ({u}, s,t, 1)  
}{|t|^{2+\a}}  
\cK_{\a, 0}(u,s,t)
d t  +   \int_{\R}
 {
 t\L ({u}, s,t, 1)
}{
\cK_{\a, 1}(u,s,t)
}
dt   \\
&  
 -2 {{u}(s)}    \int_{\R}
 {
\cK_{\a, 1}(u,s,t)
}
dt.  \nonumber   
\end{align*} 

As in Section \ref{sec:preparations-1}, to prove the regularity of $H$, it will be crucial to have estimates  
related to the maps  $\cK_{\a,0}$ and $\cK_{\a,1}$.
\begin{lemma}\label{lem:est-cK-2D}
Let $N=2$, $k \in \N \cup \{0\}$,  $\d,\a>0$ and $\b\in (0,1)$. 
\begin{enumerate}
\item[(i)] There exists a constant $ c=c(\alpha,\b,k)>1 $ such that 
 such that  for all $(s,s_1,s_2,t)\in\R^4$ and ${u}\in C^{1,\b}(\R)$, we have 
   \be 
   \label{eq:Dk-K0-s}
\|  D_{u}^k {\cK}_{\alpha,0}({u},s,t    )   \|\leq   
 {c(1+ \|{u}\|_{C^{1,\b}(\R)} )^{c}   }   ,
 \ee 
   \be  
   \label{eq:Dk-K0-s_1s_2}
\| [D_{u}^k {\cK}_{\alpha,0}({u},\cdot ,t );s_1,s_2] \|\leq   
 {c(1+ \|{u}\|_{C^{1,\b}(\R)} )^{c}   \, |s_1-s_2|^\b } .     
 \ee
\item[(ii)] There exists  $c=c(\alpha,\b,k,\d)>1$ such that  for all $(s,s_1,s_2,t)\in\R^4$ and ${u}\in \cO_\d$, we have 
$$ \label{eq:Dk-K1-s}
\|  D_{u}^k {\cK}_{\alpha,1}({u},s,t   )   \|\leq   
\frac{c(1+ \|{u}\|_{C^{1,\b}(\R)} )^{c}   }{    (1 + t^2)^{(2+\alpha)/2}    }  ,
$$
$$  \label{eq:Dk-K1-s_1s_2}
\| [D_{u}^k {\cK}_{\alpha,1}({u},\cdot ,t );s_1,s_2] \|\leq   
\frac{c(1+ \|{u}\|_{C^{1,\b}(\R)} )^{c}   \, |s_1-s_2|^\b }{    (1 + t^2)^{(2+\alpha)/2}}.     
$$
 \end{enumerate}
 \end{lemma}
\begin{proof}
The proof of  (ii) is the same as the proof of Lemma \ref{lem:est-cK}.  

The proof of (i) is very simple. Note simply that the function $y\mapsto (1+y^2)^{-(2+\a)/2}$ is a bounded smooth
function with bounded derivatives of all orders. At the same time, $\L_0(w,s,t,1)$ is a linear operator on $w$ satisfying
$| \L_0(w,s,t,1)| \leq  \|w\|_{C^{1,\b}(\R)}$ and
$| \L_0(w,s_1,t,1) -  \L_0(w,s_2,t,1)| \leq \|w\|_{C^{1,\b}(\R)}|s_1-s_2|^\b$. The claimed estimates follow easily 
from these two facts, applying the bounds for $\L_0(w,\cdot,t,1)$ at $w=u$ and/or $w=\varphi_i$ (when considering the 
$k$-th derivatives of $\cK_{\a,0}$ at $u$ in directions $[\varphi_i]$).
\end{proof}

The following two lemmas provide the desired estimates for the formal candidates to be the derivatives of $H$.
\begin{lemma}\label{lem:est-cand-deriv-2D}  
Let $N=2$, $\d>0$, ${u} \in \cO_\d$ and $\phi, {u}_1,\dots, {u}_k \in C^{1,\b}(\R)$ and $k\in \N \cup \{0\}$. We  define the functions $\cF_i: \R\to \R$ by 
\begin{align*}
\cF_0(s)=  \int_{\R}  \frac{t}{|t|^{2+\a}} \L (\phi, s,t,1)  D_{u}^k {\cK}_{\a,0}({u},s ,t    ) [u_1,\dots,u_k] dt, 
\end{align*}
\begin{align*}
\cF_1(s)=  \int_{\R}  t\L (\phi, s,t,1)  D_{u}^k {\cK}_{\a,1}({u},s ,t    ) [u_1,\dots,u_k] dt
\end{align*}
  and 
\begin{align*}
 \cF_2(s)= \int_{\R}  D_{u}^k {\cK}_{\a,1}({u}, s ,t   )[u_1,\dots,u_k]   dt   .
\end{align*}
Then $\cF_i\in C^{0,\beta-\alpha}(\R)$, for $i=0,1$  and $\cF_2 \in C^{0,\beta}(\R)$. Moreover, there exists a constant $c=c(\a,\b,k,\d)>1$ such that 
\begin{equation}
  \label{eq:est-F1-2D}
\|\cF_i\|_{C^{0,\beta-\alpha}(\R)} \leq  c(1+ \|{u}\|_{C^{1,\b}(\R)} )^{c} \|\phi\|_{C^{1,\b}(\R)}     \prod_{i=1}^k \|{u}_i\|_{C^{1, \b}(\R)} ,
\end{equation}
for $i=0,1$ and 
\begin{equation}
  \label{eq:est-F2-2D}
\|\cF_2\|_{C^{0,\beta}(\R)} \leq  c(1+ \|{u}\|_{C^{1,\b}(\R)} )^{c}      \prod_{i=1}^k \|{u}_i\|_{C^{1, \b}(\R)}. 
\end{equation}
\end{lemma}
\begin{proof} 
Throughout this proof, the letter $c$ stands for different constants greater than one and depending only on $\alpha,\b,k$ and $\d$. We define
$$
F_0(s,t):=  \frac{t}{|t|^{2+\a}}  \L (\phi, s,t,1)  D_{u}^k {\cK}_{\a,0}({u},s ,t    ) [u_1,\dots,u_k] 
$$
and 
$$
F_1(s,t):= t  \L (\phi, s,t,1)  D_{u}^k {\cK}_{\a,1}({u},s ,t    ) [u_1,\dots,u_k] .
$$
By \eqref{eq:Phis} and Lemma \ref{lem:est-cK-2D}, we have
\begin{align}
|F_i(s,t) |\le  c(1+ \|{u}\|_{C^{1,\b}(\R)} )^{c}   \frac{\min(|t|^{\b},1)    }{   |t|^{1+\a}  } 
\|\phi\|_{C^{1,\b}(\R)}\,    \prod_{i=1}^k \|{u}_i\|_{C^{1, \b}(\R)} ,  
\end{align}
for $i=0,1$ and consequently,  
\be
\label{eq:est-cF1-L-infty}
 \|\cF_i\|_{L^\infty(\R)}\leq   c(1+ \|{u}\|_{C^{1,\b}(\R)} )^{c} \|\phi\|_{C^{1,\b}(\R)}\,     \prod_{i=1}^k \|{u}_i\|_{C^{1, \b}(\R)}.
\ee

We now use \eqref{eq:uv-s_1s_2},  the estimates \eqref{eq:Phis}, \eqref{eq:Phis1s2} and Lemma \ref{lem:est-cK-2D} to get 
\begin{align*}
&|[F_i(\cdot,t); s_1,s_2]|\le  c(1+ \|{u}\|_{C^{1,\b}(\R)} )^{c} \;\times
 \nonumber \\
&\hspace{10mm}\Bigl(\frac{\min( |t|^\b    , |s_1-s_2|^\b)}{  |t|^{1+\a}   } 
+ \frac{\min(|t|^{\b},1)   |s_1-s_2|^\b }{   |t|^{1+\a}  }\Bigr)
\|\phi\|_{C^{1,\b}(\R)}\,    \prod_{i=1}^k \|{u}_i\|_{C^{1, \b}(\R)} .
\end{align*}
This leads to
\begin{align} 
\label{eq:est-Fi-s1s2}
|[ \cF_i; s_1,s_2]|&\le  \int_{\R} |[F_i(\cdot,t); s_1,s_2]|dt   \nonumber\\
 &\le c(1+ \|{u}\|_{C^{1,\b}(\R)} )^{c}\|\phi\|_{C^{1,\b}(\R)}\,   
 \prod_{i=1}^k \|{u}_i\|_{C^{1, \b}(\R)} \, \times \\
 &\qquad\qquad 
\int_{\R}\left\{  \min( |t|^\b    , |s_1-s_2|^\b)+\min(|t|^{\b} ,1)  |s_1-s_2|^\b    \right\}      |t|^{-1-\a}dt.   \nonumber 
\end{align}
Assuming $|s_1-s_2|\leq 1$, we have   
\begin{align*}
 \left\{ \int_{|t|\leq |s_1-s_2|} + \int_{|t|\geq |s_1-s_2|}\right\} 
 &\left\{  \min( |t|^\b    , |s_1-s_2|^\b)+\min(|t|^{\b} ,1)  |s_1-s_2|^\b    \right\}      |t|^{-1-\a}dt\\
  \leq& c \left(  \int_{|t|\leq |s_1-s_2|}    |t|^{\b-\a-1}    dt 
 + |s_1-s_2|^\b  \int_{|t|\geq |s_1-s_2|}         |t|^{-1-\a}dt \right) \\
 \leq & c|s_1-s_2|^{\b-\a}. 
\end{align*}
Using this in \eqref{eq:est-Fi-s1s2},  we then conclude that,  for $i=0,1$,
$$
|[ \cF_i; s_1,s_2]| \le c(1+ \|{u}\|_{C^{1,\b}(\R)} )^{c}  \|\phi\|_{C^{1,\b}(\R)}\,  |s_1-s_2|^{\beta-\alpha } \prod_{i=1}^k \|{u}_i\|_{C^{1, \b}(\R)}.
$$
This together with \eqref{eq:est-cF1-L-infty} give (\ref{eq:est-F1-2D}). 

To prove (\ref{eq:est-F2-2D}), we now set 
$$
{F_2}(s,t):= D_{u}^k {\cK}_{\a,1}({u}, s ,t    )[u_1,\dots,u_k]  ,
$$
and  by   Lemma \ref{lem:est-cK-2D},  we have 
\begin{align*}
|[{F_2}(\cdot,t); s_1,s_2]| \leq  
 c(1+ \|{u}\|_{C^{1,\b}(\R)} )^{c} \frac{ |s_1-s_2|^\b }{(1+ t^2)^{(2+\a)/2} }      \prod_{i=1}^k \|{u}_i\|_{C^{1, \b}(\R)}  
\end{align*}
and 
\begin{align*}
\|[{F_2}(\cdot,t)\|_{L^\infty(\R)} \leq  
 c \frac{(1+ \|{u}\|_{C^{1,\b}(\R)} )^{c}}{(1+ t^2)^{(2+\a)/2} }        
 \prod_{i=1}^k \|{u}_i\|_{C^{1, \b}(\R)}.
\end{align*}
We thus have (\ref{eq:est-F2-2D}), since    $\cF_2(\cdot)=\int_{\R}F_2(\cdot,t)dt$. 
 \end{proof}
 Next, we define
$$
M_0({u}, s,t)= \frac{t}{|t|^{2+\a}}\L(u,s,t, 1) {\cK}_{\a,0}(u,s,t) ,  \qquad M_1({u}, s,t)=  {t} \L(u,s,t, 1) {\cK}_{\a,1}(u,s,t)
$$
and 
$$  
M_2({u}, s,t)=-2u(s)  {\cK}_{\a,1}  (u,s,t) ,
$$
so that
 \begin{align*}
  -\frac{\alpha}{2}  H({u})  (s) =&\sum_{\ell=0}^2   \int_{\R}  M_\ell({u},s,t)  
d t.
\end{align*} 
We also recall from \eqref{eq:Dk-LT2} that, if $k \ge 1$, 
 \begin{align*}
 D^{k}_{u} M_0({u} , s,t)[u_i]_{i \in \{1,\ldots,k\}} 
 &=  \frac{t}{|t|^{2+\a}} \L({u},s,t,1)    D^{k }_{u}\cK_{\a,0}({u},s,t)[u_i]_{i \in \{1,\ldots,k\}} \\
&\hspace{8mm} +  \sum_{j=1}^k \frac{t}{|t|^{2+\a}}  \L({u}_{j},s,t,1)   D^{k-1}_{u} \cK_{\a,0}({u},s,t) 
 [u_i]_{\stackrel{i \in \{1,\ldots,k\}}{ i\neq j}},
 \end{align*}
  \begin{align*}
 D^{k}_{u} M_1({u} , s,t)[u_i]_{i \in \{1,\ldots,k\}} &=  t \L({u},s,t,1)    
 D^{k }_{u}\cK_{\a,1}({u},s,t)[u_i]_{i \in \{1,\ldots,k\}}
 \\
&\hspace{8mm} + \sum_{j=1}^k t  \L({u}_{j},s,t,1)   D^{k-1}_{u} \cK_{\a,1}({u},s,t) 
 [u_i]_{\stackrel{i \in \{1,\ldots,k\}}{ i\neq j}}
 \end{align*}
 and
 \begin{align*}
 D^{k}_{u} M_2({u} , s,t)[u_i]_{i \in \{1,\ldots,k\}} &= -2 u(s)   D^{k }_{u}\cK_{\a,1}({u},s,t)[u_i]_{i \in \{1,\ldots,k\}} \\
&\hspace{8mm} - 2  \sum_{j=1}^k u_j(s)   D^{k-1}_{u} \cK_{\a,1}({u},s,t) 
 [u_i]_{\stackrel{i \in \{1,\ldots,k\}}{ i\neq j}}.
 \end{align*} 
 
With this and the estimates in Lemma \ref{lem:est-cand-deriv-2D}, we can now follow step by step the arguments in  Section \ref{sec:preparations-1} (noticing that the proof of  Lemma \ref{prop:smooth-ovH-0} and Proposition  \ref{prop:smooth-ovH} are essentially algebraic) to deduce that $H:\cO\subset C^{1,\b}(\R)\to C^{0,\b-\a}(\R)$  is of class $C^\infty$. Moreover 
 \begin{align*}
  -\frac{\alpha}{2}  D^kH({u}) [u_1,\dots,u_k](s) =&\sum_{\ell=0}^2   \int_{\R}D^k M_\ell({u},s,t) [u_1,\dots,u_k]
d t.
\end{align*} 
 
As  remarked earlier,  in our 2D paper \cite{Cabre2015A},  we applied the implicit function theorem to  the $C^1$ map
$$
\ov{\Phi}: \R\times \R_+\times X\to Y, \qquad   \ov{\Phi} (a,\l, v):=\frac{1}{a}\{H(\l R+a(\cos(\cdot)+v))-H(\l R)\}
$$
at the point   $(0,1,0)$, where $R>0$ was chosen in such away that  $  \ov{\Phi}  (0,1, 0)=0$ and that  the linear  maps   $D_\l  \ov{\Phi}  (0, 1, 0):\R  \to \la \cos(\cdot) \ra  $ and    $D_v  \ov{\Phi}  (0, 1, 0): X_{ \perp} \to Y_{ \perp} $  are  invertible.  
Since, for every $s\in \R$,
$$
\ov{\Phi} (a,\l, v)(s)= \int_0^1D H\left( \l R+a t (\cos(\cdot)+v) \right)[\cos(\cdot)+v](s) dt,
$$
it follows that $ \ov{\Phi}$ is  of class $C^\infty$ in a neighborhood of $(0,1,0)$, for every $R>0$. 
Hence  the curves $a\mapsto \l(a)$ and $a\mapsto v_a$  that we obtained in  \cite[Theorem 1.2]{Cabre2015A} are   smooth. 

We recall that this branch in $\R^2$ could have been obtained also using the Crandall-Rabinowitz theorem as in the present
paper.
\vspace{0.2cm}

\textbf{Acknowledgement:}
The authors wish to thank the referees for their careful reading and their valuable comments.


\begin{thebibliography}{99}

\bibitem{Ambrosio} 
\newblock L. Ambrosio, G. De Philippis, L. Martinazzi,
\newblock Gamma-convergence of nonlocal perimeter functionals,
\newblock Manuscripta Math. 134 (2011), 377--403.\vspace*{2pt} 



\bibitem{Barrios} 
\newblock B. Barrios, A. Figalli, E. Valdinoci,
\newblock Bootstrap regularity for integro-differential operators and its application to nonlocal minimal surfaces,
\newblock Ann. Sc. Norm. Super. Pisa Cl. Sci. (5) 13 (2014), 609--639.\vspace*{2pt}


\bibitem{Bucur}
\newblock C. Bucur, E. Valdinoci,
\newblock Nonlocal diffusion and applications,
\newblock Lecture Notes of the Unione Matematica Italiana 20, Springer International Publishing Switzerland 2016.\vspace*{2pt} 


\bibitem{Cabre2015A} 
\newblock X. Cabr\'{e}, M. M. Fall, J. Sol\`{a}-Morales, T. Weth,
\newblock Curves and surfaces with constant nonlocal mean curvature: meeting Alexandrov and Delaunay,
\newblock J. Reine Angew. Math., Online First, DOI: 10.1515/crelle-2015-0117.\vspace*{2pt} 


\bibitem{Cabre2015B} 
\newblock X. Cabr\'{e}, M. M. Fall, T. Weth,
\newblock Near-sphere lattices with constant nonlocal mean curvature,
\newblock forthcoming. \vspace*{2pt}


\bibitem{Caffarelli2010}
\newblock L. Caffarelli, J.-M. Roquejoffre, O. Savin,
\newblock Nonlocal minimal surfaces,
\newblock Comm. Pure Appl. Math. 63 (2010), 1111--1144.\vspace*{2pt}


\bibitem{Caff-Soug2010} 
\newblock L. Caffarelli, P. E. Souganidis,
\newblock Convergence of nonlocal threshold dynamics approximations to
              front propagation,
\newblock Arch. Ration. Mech. Anal. 195 (2010), 1--23.\vspace*{2pt}


\bibitem{Ciraolo2015} G. Ciraolo, A. Figalli, F. Maggi, M. Novaga,
\newblock Rigidity and sharp stability estimates for hypersurfaces with constant
and almost-constant nonlocal mean curvature,
\newblock J. Reine Angew. Math., Online First. DOI: 10.1515/crelle-2015-0088.\vspace*{2pt}


\bibitem{Cinti.Sinestrari.Valdinoci-ArXiv}
\newblock E. Cinti, C. Sinestrari, E.Valdinoci,
\newblock Neckpinch singularities in fractional mean curvature flows,
\newblock arXiv:1607.08032v2.\vspace*{2pt} 

\bibitem{Chambolle.Morini.Ponsiglione2015}
\newblock A. Chambolle, M. Morini, M. Ponsiglione,
\newblock Nonlocal curvature flows. 
\newblock Arch. Ration. Mech. Anal. 218 (2015), 1263--1329.\vspace*{2pt}  


\bibitem{Crandall1971}
\newblock M.~G. Crandall, P. H. Rabinowitz,
\newblock Bifurcation from simple eigenvalues,
\newblock J. Functional Analysis 8 (1971), 321--340.\vspace*{2pt}



\bibitem{Davila02}
\newblock J. D\'avila,
\newblock On an open question about functions of bounded variation,
\newblock Calc. Var. Partial Differential Equations 15 (2002), 519--527.\vspace*{2pt} 

  
\bibitem{Davila2015} 
\newblock J. D\'{a}vila, M. del Pino, S. Dipierro, E. Valdinoci,
\newblock Nonlocal Delaunay surfaces,
\newblock Nonlinear Analysis: Theory, Methods \& Applications 137 (2016), 357--380.\vspace*{2pt}

  
\bibitem{Delaunay}
\newblock Ch. Delaunay,
\newblock Sur la surface de r\'{e}volution dont la courbure moyenne est constante,
\newblock J. Math. Pures Appl. 1\`ere. s\'{e}rie 6 (1841), 309--315.\vspace*{2pt}


\bibitem{GR}
\newblock I. S. Gradshteyn, I. M. Ryzhik,
\newblock Table of integrals, series, and products. Seventh edition,
\newblock Elsevier/Academic Press, Amsterdam, 2007.\vspace*{2pt} 


\bibitem{Imbert2009}
\newblock C. Imbert,
\newblock Level set approach for fractional mean curvature flows,
\newblock Interfaces Free Bound. 11 (2009), 153--176.\vspace*{2pt} 


\bibitem{FaadeBruno-JW}
\newblock W. P. Johnson,
\newblock The Curious History of Fa\'{a} di Bruno's Formula,
\newblock Am. Math. Monthly 109 (2002), 217--227.\vspace*{2pt} 


\bibitem{grisvard}
\newblock P. Grisvard,
\newblock Elliptic Problems in Nonsmooth Domains,
\newblock SIAM Classics in Applied Mathematics 69, Philadelphia, PA 2011.\vspace*{2pt} 


\bibitem{Ponce04}
\newblock A. C. Ponce, 
\newblock A new approach to Sobolev spaces and connections to $\Gamma$-convergence,
\newblock Calc. Var. Partial Differential Equations 19 (2004), 229--255.\vspace*{2pt} 


\bibitem{Saez2015}
\newblock M. S\'aez, E. Valdinoci,
\newblock On the evolution by fractional mean curvature,
\newblock arXiv:1511.06944.\vspace*{2pt}

    
\bibitem{Sicbaldi2012}
\newblock F. Schlenk, P. Sicbaldi,
\newblock Bifurcating extremal domains for the first eigenvalue of the Laplacian,
\newblock Adv. Math. 229 (2012), 602--632.\vspace*{2pt}


\bibitem{Sicbaldi2010}
\newblock P. Sicbaldi,
\newblock New extremal domains for the first eigenvalue of the Laplacian in flat tori,
\newblock Calc. Var. and PDEs 37 (2010), 329--344.\vspace*{2pt}


\bibitem{Sil}
\newblock L. Silvestre,
\newblock Regularity of the obstacle problem for a fractional power of the {L}aplace operator,
\newblock Comm. Pure Appl. Math. 60 (2007), 67--112.\vspace*{2pt}

  
  
\end{thebibliography}
\end{document}